\DocumentMetadata{uncompress}
\documentclass[11pt]{amsart}
\DeclareUnicodeCharacter{00A0}{ }
\DeclareUnicodeCharacter{3000}{ } 

\usepackage{amsfonts,graphics,amsmath,amsthm,amscd,amssymb,latexsym,multicol,mathrsfs}
\usepackage{epsfig,url}
\usepackage{flafter}
\usepackage{fancyhdr}
\usepackage[hyperindex,breaklinks]{hyperref}
\usepackage[dvipsnames]{xcolor}
\hypersetup{
	colorlinks=true,
	linkcolor=black,
	citecolor=black,
	filecolor=cyan,
	urlcolor=purple
}

\usepackage{graphicx}
\usepackage{float}
\usepackage{bm}
\usepackage{enumitem}
\usepackage{multirow}
\usepackage{scalerel,stackengine}
\usepackage{stmaryrd}
\usepackage{datetime}
\usepackage[title]{appendix}
\usepackage{tabulary}
\usepackage{booktabs}
\usepackage{tikz}
\usetikzlibrary{positioning}
\usepackage{verbatim}
\usepackage{mathtools}
\usepackage{braket}

\usepackage[all,cmtip]{xy}

\DeclareMathOperator{\Div}{Div}

\DeclareMathOperator{\red}{red}

\DeclareMathOperator{\Spec}{Spec}
\DeclareMathOperator{\Supp}{Supp}

\DeclareMathOperator{\fs}{fs}

\numberwithin{equation}{subsection}
\numberwithin{footnote}{subsection}

\newtheorem{cor}[subsection]{Corollary}
\newtheorem{lem}[subsection]{Lemma}
\newtheorem{prop}[subsection]{Proposition}
\newtheorem{thm}[subsection]{Theorem}

\newtheorem{quest}[subsection]{Question}

{
 \theoremstyle{definition}

 \newtheorem{defn-lem}[subsection]{Definition-Lemma}

}

\newcommand{\N}{\mathbb N}
\newcommand{\PP}{\mathbb P}

\newcommand{\Q}{\mathbb Q}
\newcommand{\R}{\mathbb R}
\newcommand{\Z}{\mathbb Z}

\newcommand{\bir}{\dashrightarrow}
\newcommand{\rddown}[1]{\left\lfloor{#1}\right\rfloor} 

\newcommand{\mni}{\medskip\noindent}
\newcommand{\wt}{\widetilde}
\newcommand{\wh}{\widehat}
\newcommand{\abs}[1]{\left\vert#1\right\vert}
\newcommand{\centre}{\operatorname{centre}}
\newcommand{\irr}{\operatorname{irr}}
\newcommand{\gon}{\operatorname{gon}}

\newcommand{\mbb}{\mathbb}

\newcommand{\mc}{\mathcal}
\newcommand{\mf}{\mathfrak}

\title{\large I\MakeLowercase{rrationality of degenerations of} F\MakeLowercase{ano varieties}}
\author{\large C\MakeLowercase{aucher} B\MakeLowercase{irkar} \MakeLowercase{and} \large S\MakeLowercase{antai} Q\MakeLowercase{u}}
\date{\today}

\begin{document}

\begin{abstract}
    In this paper we investigate the degrees of irrationality of degenerations of $\epsilon$-lc Fano varieties
    of arbitrary dimensions.
    We show that given a generically $\epsilon$-lc klt Fano fibration 
    $X\to Z$ of dimension $d$ over a smooth curve $Z$
    such that $(X, t F)$ is lc for a positive real number $t$, where $F$
    is the reduction of an irreducible central fibre of $X$ over a closed point $z\in Z$, then $F$
    admits a rational dominant map $\pi\colon F\bir C$ 
    to a smooth projective variety $C$ with bounded degree of irrationality
    depending only on $d, \epsilon, t$
    such that the general fibres of $\pi$ are irreducible and rational.  
    This proves the generically bounded case of 
    a conjecture proposed by the first author and Loginov
    for log Fano fibrations of dimensions greater than three.
    One of the key ingredients in our proof is to modify the generically $\epsilon$-lc klt Fano fibration $X\to Z$
    to a toroidal morphism of toroidal embeddings with bounded general fibres.
\end{abstract}

\maketitle

\tableofcontents

\addtocontents{toc}{\protect\setcounter{tocdepth}{1}}


\section{\bf Introduction}

We work over an algebraically closed field $\mbb K$ of characteristic zero unless stated otherwise.

\mni
{\textbf{\sffamily{Bounded irrationality of divisors on Fano fibrations.}}}
Given a family of varieties parameterised by a smooth curve, 
degenerations are the limit varieties of this family of varieties.
Many geometric properties of varieties are preserved by passing to degenerations.
For example, irreducible degenerations of uniruled varieties (respectively, rationally chain connected varieties) 
are also uniruled (respectively, rationally chain connected); see \cite[Theorem~IV.1.8, Theorem~IV.3.5.3]{Kol-rc}.
However, it is well-known that rational varieties can degenerate to non-rational varieties
even for rational del Pezzo surfaces; for instance, 
a smooth cubic surface can degenerate to a cone over an elliptic curve which is a
non-rational and singular del Pezzo surface.
This phenomenon is closely related to the fact that the degeneration has 
log canonical (lc) but not Kawamata log terminal (klt) singularities.

In the paper \cite{BirkarLoginov}, the first named author and Loginov study 
the non-rationality property of degenerations of
klt del Pezzo surfaces and prove how far the irreducible components of the degenerations
can be from being rational.  More precisely, 
the boundedness of irrationality of degenerations
is proved in \cite{BirkarLoginov} as follows.
In particular, by \cite[Theorem~1.1]{BirkarLoginov}, the irrationality of degenerations of klt del Pezzo surfaces
is bounded in a certain way even though the degenerations do not belong to a bounded family of surfaces,
that is, the collection of degenerations can not be parameterised by a finite type scheme; see \S \ref{b-bnd-couples}.
Recall that a contraction $f\colon X\to Z$ of normal varieties is called a
\emph{Fano fibration} (respectively, \emph{klt Fano fibration})
if $X$ has lc singularities (respectively, klt singularities)
and $-K_X$ is ample over $Z$; see \S \ref{defn-fano-fibr}.

\begin{thm}[\protect{\cite[Theorem~1.1]{BirkarLoginov}}]\label{Birkar-Loginov}
    Fix a positive real number $t$.  Assume that $f\colon X\to Z$ is a klt Fano fibration
    where $\dim X = 3$ and $\dim Z =1$.  Assume that $F$ is the reduction of an irreducible fibre 
    and that $(X, t F)$ is lc.  Then:
    \begin{itemize}
    	\item [\emph{(i)}] $F$ is birational to $\PP^1\times C$, where $C$ is a smooth projective curve 
    	                   with gonality $\gon (C)$ bounded depending only on $t$,
    	\item [\emph{(ii)}] if $t>\frac{1}{2}$, then the genus $g(C)$ is bounded, and
    	\item [\emph{(iii)}] if $t=1$, then the genus $g(C)\le 1$.
    \end{itemize}
\end{thm}

Note that the genus $g(C)$ can be arbitrarily large if $t\le \frac{1}{2}$; see \cite[Example~2.3]{BirkarLoginov}.
Moreover, Theorem~\ref{Birkar-Loginov} follows from a more general result for 
Fano type log Calabi-Yau fibrations, where the base $Z$ of the fibration can have
dimension $\ge 1$; see \cite[Theorem~1.3]{BirkarLoginov}.
Recall that the \emph{gonality} $\gon (C)$ of a smooth curve $C$ is the smallest degree of dominant morphisms
of curves $C\to \PP^1$.  In higher dimensions, the \emph{degree of irrationality} $\irr (D)$ of a variety $D$
is the least possible degree of dominant rational maps $D\dashrightarrow \PP^{\dim D}$; see \S \ref{defn-irr}.
The first named author and Loginov ask in \cite{BirkarLoginov}
whether a similar result to Theorem~\ref{Birkar-Loginov} 
for klt Fano fibrations of higher dimensions should also hold.

\begin{quest}[cf. \protect{\cite[Question~1.4]{BirkarLoginov}}]\label{quest-bnd-irra}
	Fix a positive real number $t > 0$ and a natural number $d$. Suppose that
    $f\colon X\to Z$ is a klt Fano fibration over a smooth curve $Z$, where $\dim X = d$. Assume that $D$
    is the reduction of an irreducible fibre of $f$ such that $(X, tD)$ is lc. Is it true that there is a
    rational map $D\dashrightarrow C$, where the general fibres are rationally connected and $C$ is a smooth
    projective variety with bounded degree of irrationality?
\end{quest}

In this paper, we apply a very different approach from \cite{BirkarLoginov}
via toroidal geometry to prove the boundedness of degrees of irrationality 
for degenerations of klt Fano fibrations of arbitrary dimensions
with $\epsilon$-lc general fibres; cf. \cite[\S 5, \S 6]{BirkarLoginov}.

\begin{thm}\label{bnd-irra}
    Fix positive real numbers $\epsilon, t>0$ and a natural number $d$.
    Assume that $f\colon X\to Z$ is a klt Fano fibration with $\dim X =d$ such that
    \begin{itemize}
    	\item [\emph{(1)}] $Z$ is a smooth curve,
    	\item [\emph{(2)}] $X$ is $\epsilon$-lc over the generic point of $Z$, and
    	\item [\emph{(3)}] $F$ is the reduction of an irreducible closed fibre of $f$ and $(X, t F)$ is lc.
    \end{itemize}
    Then there is a dominant rational map $F \dashrightarrow C$ whose general fibres are irreducible and rational,
    and $C$ is a bounded smooth projective variety
    depending only on $d, \epsilon, t$,
    hence with bounded degree of irrationality.
\end{thm}

Note that the general fibres of $F\dashrightarrow C$ in Theorem~\ref{bnd-irra} are \emph{rational},
which is a much stronger geometric property than the rational connectivity 
predicted in Question~\ref{quest-bnd-irra}.  Example~2.1 of \cite{BirkarLoginov} 
shows that Theorem~\ref{bnd-irra} does not hold without assuming that $(X, tF)$ is lc.
A similar result can be formulated when the fibre containing $F$ is not irreducible 
by using Fano type log Calabi-Yau fibrations; cf. \cite[Theorem~1.3, Question~1.4]{BirkarLoginov}.

As a corollary of Theorem~\ref{bnd-irra}, we show that every irreducible component of fibres of $\epsilon$-lc 
Fano fibrations admits the structure as in Theorem~\ref{bnd-irra}.
Note that in this case, we do not prescribe a positive real number $t>0$ such that $(X, tF)$ is lc.

\begin{thm}\label{bnd-irra-cor}
    Let $\epsilon$ be a positive real number, and let $d$ be a natural number. 
    Let $f\colon X\to Z$ be a Fano fibration with $\dim X = d$, where $Z$ is a smooth curve,
    and $X$ is $\epsilon$-lc.  Let $F$ be an irreducible component of a closed fibre of $f$.
    Then there is a dominant rational map $F \dashrightarrow C$ whose general fibres are irreducible and rational,
    and $C$ is a bounded smooth projective variety depending only on $d, \epsilon$,
    hence with bounded degree of irrationality.
\end{thm}


\mni
{\textbf{\sffamily{Relative toroidalisation and sketch of the proof.}}}
One of the key ingredients in our proof of
Theorem~\ref{bnd-irra} is that we can modify the family $f\colon X\to Z$
to a generically bounded toroidal morphism of strict toroidal embeddings.
A general fibre of $X\to Z$ in Theorem~\ref{bnd-irra} is an $\epsilon$-lc Fano variety
which is bounded by \cite[Theorem~1.1]{B-BAB}.
Thus, there is a birational map $\phi\colon X\dashrightarrow Y$ over $Z$
such that every fibre of $g\colon Y\to Z$ is bounded; indeed, $Y\to Z$
is relatively bounded (see \S \ref{r-bnd-def}).
Moreover, $\phi$ can be chosen so that it does not contract any curve over the generic point of $Z$ as
general fibres of $f$ are already bounded.
In \S \ref{pf-bnd-irra}, we show that there exists a toroidal morphism $f'$ fitting into the commutative diagram
\[\xymatrix{
(U_{Y'}\subset Y')\ar[d]_{f'}\ar[r]^-{m_X} & Y\ar[d]^g & X\ar[d]^f\ar@{-->}[l]_-{\phi} \\
(U_{Z'}\subset Z')\ar[r]_-{m_Z} & Z\ar@{=}[r] & Z
}\]
where the left vertical arrow $f'$ is a toroidal morphism of strict toroidal embeddings (see \S \ref{toroidal-couples-defn}
and \S \ref{toroidal-morphisms}),
$m_X$ and $m_Z$ are projective birational morphisms, and the general fibres of $f'$ are bounded.

Note that the toroidal modification exists
for any projective surjective morphism of varieties
by \cite{ADK13, weak:s-stable:reduction}; 
see Theorem~\ref{ak00-thm}.  However, here we require a stronger condition 
on the boundedness of general fibres of $f'$,
which will be deduced from Theorem~\ref{functorial-toroidalization}.
Similar constructions are applied in \cite{birkar2025toroidaltoricmodelsfibrations},
where $m_X, m_Z$ may not be birational.
The techniques related to toroidal embeddings are also used in \cite{jiao25}.

Denote by $C'$ the centre of $F$ on $Y'$.  
By decreasing $t$ if necessary, we can assume that $t$ is a positive rational number.
As will be shown in \S \ref{pf-bnd-irra}, 
by taking an $n$-complement of $K_X + tF$
for some $n\in \N$ depending only on $d, t$ (see \S \ref{complements}),
the toroidal modification $f'$ can be taken so that
these additional properties are also satisfied:
\begin{itemize}
	\item [(i)] the support of the fibre of $f'$ over $z\in Z\cong Z'$ is contained in the 
	            toroidal boundary $D' \coloneqq Y'\setminus U_{Y'}$, in particular, $C'$ is contained in $D'$, and 
	\item [(ii)] the log discrepancy $a(F, Y', D')$ is equal to zero.
\end{itemize}
Denote by $S$ the reduction of the fibre of $f'$ over $z\in Z\cong Z'$.
Then $C'$ is an lc centre of $(Y', D')$ contained in $S$.  
By taking a sufficiently ample$/Z'$ divisor on $Y'$, we are in the situation to apply
the results in \cite{Birkar-moduli-polarized, HMX13-auto-groups, HMX14} 
to show that the lc centre $C'$ is birationally bounded; cf. Lemma~\ref{b-bnd} and \S \ref{pf-bnd-irra}.

Let $C$ be a bounded nonsingular projective variety that is birational to $C'$. 
Then there is a rational map $\pi\colon F\dashrightarrow C$.
Now $C$ is a smooth projective variety with bounded degree of irrationality since
it is bounded; cf. Lemma~\ref{irr-bnd-above}.
Moreover, as $C'$ is an lc centre of the strict toroidal couple $(Y', D')$, a general fibre of $\pi$ 
is irreducible and rational; see Proposition~\ref{rational-fibres}.
This concludes Theorem~\ref{bnd-irra}.

The following result compares Theorem~\ref{bnd-irra} and \cite[Theorem~1.1]{BirkarLoginov} when $\dim X = 3$.

\begin{cor}[cf. \protect{\cite[Theorem~1.1]{BirkarLoginov}}]\label{genus-dim-3}
	Fix positive real numbers $\epsilon, t$.  Let $f\colon X\to Z$ be a klt Fano fibration,
    where $\dim X = 3$ and $\dim Z =1$ such that
    $X$ is $\epsilon$-lc over the generic point of $Z$.  Let $F$ be the reduction of an irreducible closed fibre of $f$
    such that $(X, t F)$ is lc.
    Then $F$ is birational to $\PP^1\times E$, where $E$ is a smooth projective curve with gonality $\gon (E)$ 
    and genus $g(E)$ bounded depending only on $\epsilon$ and $t$.
\end{cor}

\begin{proof}
	From the discussion as above, there is a rational map $\pi\colon F\dashrightarrow C$ such that
	\begin{itemize}
		\item [(i)] $C$ is a bounded smooth projective variety, and
		\item [(ii)] a general fibre of $\pi$ is irreducible and rational.
	\end{itemize}
	Then the dimension of $C$ gives several possibilities for the structure of $F$:
    \begin{itemize}
	    \item [(1)] if $\dim C = 0$, then $F$ is a rational surface,
	    \item [(2)] if $\dim C = 1$, then $C$ has bounded gonality $\gon (C)$ and genus $g(C)$, 
	            and a general fibre of $\pi$ is isomorphic to $\PP^1$,
	            hence $F$ is birational to $\PP^1\times C$, and
	    \item [(3)] if $\dim C = 2$, then $F$ is birationally bounded.
    \end{itemize}
    Note that $F$ must be ruled as it is a degeneration of rational surfaces
    (see \cite[Theorem~IV.1.6]{Kol-rc}), so $F$ is birational to $\PP^1\times E$
    for some smooth projective curve $E$.  In case (1), we can take $E = \PP^1$.
    In case (2), $E$ is isomorphic to $C$ which has bounded gonality and genus.
    Now we show that $E$ also has bounded gonality and genus in case (3).
    As $F$ is birationally bounded in this case, $F$ has bounded degree of irrationality by Lemma~\ref{irr-bnd-above},
    hence $E$ also has bounded gonality which is bounded from above by the degree of irrationality of $F$;
    see \cite[Remark~2.1]{gon-surface}.  
    On the other hand, the irregularity of the surface $\PP^1\times E$ is equal to the genus $g(E)$;
    see \cite[Exercise~III.8.4]{Hart}.
    Since $\PP^1\times E$ is birationally bounded, there are only finitely many 
    possible values for the irregularity of $\PP^1\times E$, so
    $g(E)$ is also bounded from above.
\end{proof}


\mni
{\textbf{\sffamily{Boundedness of $\epsilon$-lc Fano varieties.}}}
Keeping the same notation as in Theorem~\ref{bnd-irra} and assuming that $\dim X = 3$,
we know that $F$ is birational to $\PP^1\times E$ for some smooth projective curve $E$.
Notice that the assumptions in Theorem~\ref{bnd-irra} 
for $\dim X = 3$ are stronger than those in Theorem~\ref{Birkar-Loginov}
by requiring additionally that $X$ is $\epsilon$-lc over the generic point of $Z$,
that is, a general fibre of $X\to Z$ is an $\epsilon$-lc del Pezzo surface for the fixed $\epsilon > 0$.
It is well-known that $\epsilon$-lc del Pezzo surfaces are bounded for any fixed $\epsilon>0$
(cf. \cite{Alex-K2}), and the family of del Pezzo surfaces is not bounded for $\epsilon = 0$.
This leads to the phenomenon that both the gonality and genus of $E$ are always bounded
under the conditions of Theorem~\ref{bnd-irra}; see Corollary~\ref{genus-dim-3}.
However, dropping the boundedness condition, \cite[Example~2.3]{BirkarLoginov} illustrates
that the genus of the curve $E$ can be arbitrarily large.
A general fibre of the klt Fano fibration constructed in \cite[Example~2.3]{BirkarLoginov} is 
isomorphic to the weighted projective space $\PP(1,1,n)$ with $n\ge 3$.
The log discrepancy of $\PP(1, 1, n)$ is equal to $2/n$, hence
the family of del Pezzo surfaces $\PP (1, 1, n)$ with $n \ge 3$ is not bounded; cf. \cite[Example~1.2]{B-BAB}.
The boundedness condition on general fibres of the klt Fano fibration
is crucial in our proof of Theorem~\ref{bnd-irra} in higher dimensions, that is,  
$\epsilon$-lc Fano varieties of dimension $d$ form a bounded family
for any fixed positive real number $\epsilon > 0$ and natural number $d$;
see \cite[Theorem~1.1]{B-BAB}.
It is also evident from the discussion above that boundedness of complements
is another crucial ingredient of our proofs; see \cite{B-Fano}.


\mni
{\textbf{\sffamily{Acknowledgements.}}}
We thank Professor Dan Abramovich for answering our numerous questions about the results
in \cite{ATW20}.  We would like to thank Dingxin Zhang and Heer Zhao 
for very helpful discussions about logarithmic geometry.
We would also like to thank Mao Sheng for his helpful comments.
We thank Hexu Liu for reading carefully a draft version of this paper,
and for his many suggestions to improve the paper.
We thank the anonymous referees for many helpful comments to improve this paper.
The first author was supported by a grant from Tsinghua University and 
a grant of the National Program of Overseas High Level Talent.
The results and proofs of this paper were presented by the second author
in a talk at the online Tsinghua Algebraic Geometry Seminar in October 2022.


\section{\bf Preliminaries}\label{preliminaries}

By a \emph{scheme}, we mean a separated scheme of finite type over the fixed algebraically closed field 
$\mathbb{K}$ of characteristic zero.
A \emph{variety} is an irreducible, reduced, and quasi-projective scheme;
when emphasising the quasi-projectivity, we also call it a \emph{quasi-projective variety}.
For a scheme $X$, we denote by $X_{\text{red}}$ the maximal reduced closed subscheme
of $X$ and call it the \emph{reduction of $X$}. 

Let $f\colon X\to Z$ be a morphism between schemes, where $Z$ is irreducible.
If $X$ has a unique irreducible component $Y$ that dominates $Z$, we call
$Y$ the \emph{main component of $X$}.
If this is the case, we usually just say that \emph{$Y$ is the main component of $X$}
without mentioning that it is the unique irreducible component of $X$ dominating $Z$.

\subsection{Contractions}
 
A \emph{contraction} is a projective morphism of schemes $f\colon X\to Y$
such that $f_*\mc{O}_X = \mc{O}_Y$; $f$ is not necessarily birational.
In particular, $f$ has connected fibres.  Moreover, if $X$ is normal, then $Y$ is also normal.


\subsection{Divisors}

Let $X$ be a scheme.  By a \emph{divisor}, we mean a Weil divisor on $X$, that is,
a finite $\Z$-linear combination of codimension one irreducible and reduced closed subschemes of $X$.
By a \emph{$\Q$-divisor} (respectively, an \emph{$\R$-divisor}), we mean a finite linear combination
$\sum_i b_i B_i$, where every $B_i$ is a prime Weil divisor on $X$ and $b_i\in \Q$ 
(respectively, $b_i\in \R$). 
A $\Q$-divisor (respectively, an $\R$-divisor) is called \emph{$\Q$-Cartier}
(respectively, \emph{$\R$-Cartier})
if it is a $\Q$-linear (respectively, an $\R$-linear) combination of Cartier divisors.
Let $B_1$ and $B_2$ be two $\R$-divisors on $X$.  We say that $B_1\sim B_2$ (respectively, 
$B_1\sim_{\Q} B_2$, respectively, $B_1\sim_{\R} B_2$) if $B_1 - B_2$ is a $\Z$-linear 
(respectively, a $\Q$-linear, respectively, an $\R$-linear) combination of principal divisors.

Let $f\colon X\to Z$ be a morphism of schemes, and let $L$ and $M$ be $\R$-divisors
on $X$.  We say that \emph{$L\sim M$ over $Z$} 
(respectively, \emph{$L\sim_{\Q} M$ over $Z$}, respectively, \emph{$L\sim_{\R} M$ over $Z$})
if there is a Cartier (respectively, a $\Q$-Cartier, respectively, an $\R$-Cartier) divisor $N$ on $Z$ 
such that $L-M\sim f^*N$ (respectively, $L-M \sim_{\Q} f^*N$, respectively, $L-M \sim_{\R} f^*N$);
see \cite[\S 2.3]{B-Fano}.  In this case, we usually write
$L\sim M/Z$ (respectively, $L\sim_{\Q}M/Z$, respectively, $L\sim_{\R}M/Z$).

Let $f\colon X\to Z$ be a morphism of schemes and $D$ a nonzero $\R$-divisor on $X$.
We say that $D$ is \emph{vertical$/Z$} if $f(\Supp D)$ does not contain
any generic point of $Z$.
If $D$ does not have any vertical$/Z$ irreducible components,
we say that $D$ is \emph{horizontal$/Z$}.

For the \emph{volume} of a big $\R$-divisor, we refer the readers to \cite[\S 2.2.C]{positivity-rob-I}  
for the definition and more details.


\subsection{Singularities of pairs}

In this paper, we will use standard notions and results from minimal model program; see \cite{BCHM, km98}.
Here we collect some of the most fundamental definitions for clarification.
A \emph{pair} $(X, B)$ consists of a normal quasi-projective variety $X$
and an $\R$-divisor $B$ with coefficients in $[0, 1]$ such that $K_X + B$ is $\R$-Cartier;
in this case, we say that $B$ is a \emph{boundary}.
Let $W\to X$ be a log resolution of a pair $(X, B)$, and let 
$K_W + B_W$ be the pullback of $K_X + B$.
Denote by $\mu_D B_W$ the coefficient of $B_W$ at a prime divisor $D$ on $W$.
Then the \emph{log discrepancy} of $D$ with respect to $(X, B)$
is $1-\mu_D B_W$, and it is denoted by $a(D, X, B)$.
We say that $(X, B)$ is \emph{lc} (respectively, \emph{klt}, respectively, \emph{$\epsilon$-lc})
if $a(D, X, B)$ is $\ge 0$ (respectively, $>0$, respectively, $\ge \epsilon$) 
for every prime divisor $D$ on an arbitrary log resolution $W\to X$.
Note that $a(D, X, B)$ can also be defined in the same way when 
the coefficients of $B$ do not belong to $[0, 1]$ as long as $K_X + B$ is $\R$-Cartier.

Let $(X, B)$ be an lc pair.  An \emph{lc place} of $(X, B)$ is a prime divisor $D$ on 
some birational model of $X$ such that $a(D, X, B) = 0$.
An \emph{lc centre} is the closure of the image of an lc place on $X$. 

For \emph{dlt pairs}, we refer 
the readers to \cite[\S 2.3]{km98} for the basic definitions.
Let $(X, B)$ be an lc pair.  A \emph{$\Q$-factorial dlt model} of $(X, B)$ is 
a $\Q$-factorial dlt pair $(Y, B_Y)$
admitting a projective birational morphism $\pi\colon Y\to X$ such that
$K_Y + B_Y = \pi^*(K_X + B)$ and that every $\pi$-exceptional prime divisor is 
an irreducible component of $\rddown{B_Y}$; cf. \cite[\S 5.1]{BirkarLoginov}.
In particular, see \cite[Theorem~3.1]{KK10} 
for the existence of $\Q$-factorial dlt models.

For the basic definitions and
properties of \emph{generalised pairs}, we refer the readers to \cite[\S 2]{B-Fano}
and \cite[\S 4]{B-Zhang}.
For the definition and existence of \emph{$\Q$-factorial generalised dlt models},
see \cite[\S 2.13]{B-Fano}.

Singularities of pairs, such as lc singularities, can also be defined for 
demi-normal schemes, and the corresponding pairs with lc property are called \emph{slc pairs}.
We refer the readers to \cite[Chapter~5]{Kol_singularities_of_MMP} for more details about slc pairs.


\begin{lem}\label{log-discrep-unchanged}
	Let $X$ and $Y$ be normal varieties over a variety $Z$, 
	and let $\phi\colon Y\dashrightarrow X$ be a birational map over $Z$.
	Let $B$ be an $\R$-divisor on $X$ such that $K_X + B$ is $\R$-Cartier,
    and let $B_Y$ be the $\R$-divisor on $Y$ defined by $K_Y + B_Y = \phi^*(K_X + B)$.
	Let $F$ be a prime divisor over $X$.
	Assume that $K_X + B \sim_{\R} 0 /Z$.
	Then the log discrepancy $a(F, X, B)$ is equal to $a(F, Y, B_Y)$.
\end{lem}

\begin{proof}
    Denote by $f, g$ the morphisms $X\to Z, Y\to Z$ respectively.
    Let $p\colon W\to X, q\colon W\to Y$ be a common resolution 
    of $X, Y$ such that $f\circ p = g\circ q$.
    Then $\phi^*(K_X + B)$ is defined to be the $\R$-divisor $q_*p^*(K_X + B)$.
    It is easy to see that $\phi^*(K_X + B)$ does not depend on the choice of $p,q$,
    hence the $\R$-divisor $B_Y \coloneqq \phi^*(K_X + B) - K_Y$ is well-defined.
    Replacing $W$ by higher resolutions, we can assume that 
    $p$ is a log resolution of $(X, B)$, $q$ is a log resolution of $(Y, B_Y)$, and $F$ is a prime divisor on $W$.
    
    By assumption, there is an $\R$-Cartier $\R$-divisor $N$ on $Z$ such that 
    $K_X + B = f^*N + \sum a_i \Div_X(\alpha_i)$,
    where $a_i\in \R$, $\alpha_i\in K(X)$, 
    and $\Div_X(\alpha_i)$ is the principal divisor on $X$ associated to $\alpha_i$.  
    As $\phi$ is birational, we can view each $\alpha_i$ as a rational function on $Y$ (also on $W$).
    Then it is evident that
    \[ p^*(K_X + B) = (f\circ p)^*N + \sum a_i \Div_W(\alpha_i) = (g\circ q)^*N + \sum a_i \Div_W(\alpha_i), \]
    which implies $K_Y + B_Y = g^* N + \sum a_i \Div_Y(\alpha_i)$.
    In particular, $K_Y + B_Y$ is $\R$-Cartier, and $q^*(K_Y + B_Y) = p^*(K_X + B)$.
    Then we can conclude that $a(F, X, B) = a(F, Y, B_Y)$.
\end{proof}


\subsection{Fano fibrations}\label{defn-fano-fibr}

Let $(X, B)$ be an lc pair over a variety $Z$.  We say that $(X, B)$
is \emph{log Fano over $Z$} if $-(K_X + B)$ is ample over $Z$.
If $Z$ is a point, then $(X, B)$ is called a \emph{log Fano pair}.
In this case, if $B=0$, then $X$ is called a \emph{Fano variety}.

A morphism $(X, B) \to Z$ from an lc pair $(X, B)$ is a
\emph{log Fano fibration} if $(X, B)$ is log Fano over $Z$ and
the underlying morphism $X\to Z$ is a contraction;
moreover, if $B = 0$, we call $X\to Z$ a \emph{Fano fibration}.
A \emph{klt log Fano fibration} $(X, B)\to Z$ is a log Fano fibration with 
$(X, B)$ klt.
In this case, if $B=0$, we call $X\to Z$ a \emph{klt Fano fibration};
see \cite[\S 3]{BirkarLoginov}.


\subsection{Complements}\label{complements}

We define complements as in \cite{Sh-surface}; see also \cite[\S 2.18]{B-Fano}.
Let $(X, B)$ be a pair, and let $X\to Z$ be a contraction. 
Let $T=\rddown {B}$ and $\Delta = B- T$,
and $n\in \N$ a natural number.
An \emph{$n$-complement} of $K_X + B$ over a point $z\in Z$ is of the form $K_X + B^+$
such that over some neighbourhood of $z$, we have the following properties:
\begin{itemize}
	\item $(X, B^+)$ is lc,
	\item $n(K_X + B^+)\sim 0$, and
	\item $nB^+ \ge nT+ \rddown{(n+1)\Delta}$.
\end{itemize}
In particular, $n B^+$ is an integral divisor.  
The $n$-complement is \emph{monotonic} if $B^+\ge B$.


\subsection{Couples} 

A \emph{couple} $(X,D)$ consists of a quasi-projective variety $X$ and a reduced Weil divisor $D$ on $X$.
This is more general than the definition given in \cite[\S 2.19]{B-Fano} because we are not assuming $X$ to be normal 
or projective. Also note that a couple is not necessarily a pair in the sense that we are not assuming 
$K_X+D$ to be $\R$-Cartier. In this paper, we often consider a couple $(X,D)$ equipped with a 
\emph{surjective} projective morphism to a variety $Z$
in which case we often denote the couple as $(X/Z,D)$ or $(X, D)\to Z$;
if $Z$ is a point, we call $(X, D)$ a \emph{projective couple}.

We say a couple $(X/Z,D)$ is \emph{flat} if both $X\to Z$ and $D\to Z$ are flat.


\subsection{Strata and log smooth morphisms}\label{log-smooth}

We follow the conventions as in \cite[\S 2.1]{HMX14}.
Let $(X, D)$ be a couple.  The \emph{strata} of $(X, D)$ are the irreducible components of the intersections
\[ D_I = \bigcap_{j\in I} D_j = D_{i_1}\cap \cdots \cap D_{i_r} \]
of irreducible components of $D$, where $I = \{ i_1, \dots, i_r\}$ is a subset of the indices, including the empty
intersection $X = D_{\emptyset}$.  
Every irreducible component in the strata of $(X, D)$ is called a \emph{stratum} of $(X, D)$.
If $(X, B)$ is a pair, then the \emph{strata} of $(X, B)$ are the strata of the underlying couple $(X, D)$,
where $\Supp D = \Supp B$.

If we are given a couple $(X, D)$ over a variety $T$, then we say that $(X, D)$ is \emph{log smooth over $T$}
if $(X, D)$ has simple normal crossings and the strata of $(X, D)$ are smooth over $T$.
If $(X, B)$ is a pair over a variety $U$, we say that $(X, B)$ is \emph{log smooth over $U$}
if the underlying couple $(X, D)$ is log smooth over $U$, where $\Supp D = \Supp B$.


\subsection{Bounded and birationally bounded families of couples}\label{b-bnd-couples}

We say that a set $\mc{Q}$ of projective couples $(X, D)$ is \emph{bounded}
if there exist finitely many projective morphisms $V_i\to T_i$ 
of varieties and reduced divisors $C_i$ on $V_i$ such that for each
$(X, D)\in \mc{Q}$, there exist an $i$, a closed point $t\in T_i$,
and an isomorphism of couples $\phi\colon (V_{i,t}, C_{i, t})\to (X, D)$,
where $V_{i, t}$ and $C_{i, t}$ are the fibres over $t$ of the morphisms $V_i\to T_i$ and $C_i\to T_i$ respectively.
In particular, if $D=0$ for all $(X, D)\in \mc{Q}$, we say the family $\mc{Q}$ consisting of projective varieties $X$
is \emph{bounded}.

Let $\mc{P}$ be a set of projective varieties.  We say that $\mc{P}$ is \emph{birationally bounded}
if there is a bounded family $\mc{Q}$ of projective varieties such that for every $X\in \mc{P}$,
there exists a birational map $X\dashrightarrow Y$ to a member $Y$ in $\mc{Q}$.
Note that our definition of birational boundedness is weaker than \cite[\S 2.19]{B-Fano}.

Let $\mc{Q}$ be a bounded family of projective couples,
and let $(X, D)\in \mc{Q}$.
When there is no confusion in the context, we usually say that $(X, D)$
belongs to a bounded family of projective couples, or 
just $(X, D)$ is bounded.
Similarly, if $\mc{P}$ is a birationally bounded family of projective varieties and $X\in \mc{P}$,
we usually say that $X$ belongs to a birationally bounded family of projective varieties, 
or just $X$ is birationally bounded.

We say a set $\mc{S}$ of \emph{projective pairs} $(X, B)$ is 
\emph{bounded} if there exist finitely many projective morphisms $Z_i\to U_i$,
where $U_i$ is smooth, $Z_i$ is flat over $U_i$, and pairs $(Z_i, \Sigma_i)$,
where the support of $\Sigma_i$ contains neither an irreducible component 
of a fibre nor a codimension one singular point of any fibre,
such that for every
$(X, B)\in \mc{S}$, there exist an $i$, a closed point $u\in U_i$,
and an isomorphism of pairs $\psi\colon (Z_{i,u}, \Sigma_{i, u})\to (X, B)$,
where $Z_{i,u}$ and $\Sigma_{i, u}$ are the fibres over $u$ of the morphisms $Z_i\to U_i$
and $\Sigma_i\to U_i$ respectively.
In particular, the coefficients of $B$ belong to a finite set; cf. \cite[page~873]{HMX14}.


\begin{lem}\label{lc-centre-bnd}
 	Let $\mc{Q}$ be a bounded set of lc projective pairs $(X, B)$.
 	Then the set of lc centres of $(X, B)\in \mc{Q}$ is also bounded.
\end{lem}

\begin{proof}
    Let $(X, B)\in \mc{Q}$.
    By assumption, there exist finitely many projective morphisms $Z_i\to U_i$ of varieties,
    where $U_i$ is smooth, $Z_i$ is flat over $U_i$, and pairs $(Z_i, \Sigma_i)$
    such that there exist an $i$, a closed point $u\in U_i$,
    and an isomorphism of pairs $(Z_{i,u}, \Sigma_{i, u})\to (X, B)$.
    Let $\pi_i\colon W_i\to Z_i$ be a log resolution of $(Z_i, \Sigma_i)$, and write
    \[ K_{W_i} + \Gamma_i = \pi_i^*(K_{Z_i} + \Sigma_i). \]
    Then there exists a nonempty open subset $V_i\subseteq U_i$ 
    such that for every closed point $v\in V_i$, $(Z_{i,v}, \Sigma_{i,v})$ is a pair, 
    the induced morphism of fibres
    $\pi_{i,v}\colon W_{i, v}\to Z_{i,v}$ is a log resolution of $(Z_{i,v}, \Sigma_{i,v})$, 
    and that $\pi_{i,v}^*(K_{Z_{i,v}} + \Sigma_{i,v})$ 
    is the restriction of $\pi_i^*(K_{Z_i} + \Sigma_i)$ to $W_{i,v}$.
    The complement $U_i\setminus V_i$ is a union of varieties 
    with dimensions strictly less than $\dim U_i$.  By induction on $\dim U_i$,
    we can assume that $u\in V_i$.  Then every lc centre of $(Z_{i,u}, \Sigma_{i,u})$
    is the image of a stratum of $\rddown{\Gamma_{i,u}}$ via $\pi_{i,u}\colon W_{i,u}\to Z_{i,u}$.
    Thus, there exist finitely many integral subschemes of $Z_i\times_{U_i} V_i$ such that 
    every lc centre of $(X,B)$ is isomorphic to the fibre of one of these subschemes over $u\in V_i$.
    Hence the set of lc centres of $(X, B)$ is bounded.
\end{proof}


The result as follows will be used in the proof of Theorem~\ref{bnd-irra} in \S \ref{main-proof}.

\begin{lem}\label{b-bnd}
    Let $d$ be a natural number.
    Let $(X, B + M)$ be a projective generalised lc generalised pair with $\dim X = d$ such that
    $L = K_X + B + M$ is a nef and big Cartier divisor.
    Then $\abs{m L}$ defines 
    a birational map for some $m\in \N$ depending only on $d$.
\end{lem}

\begin{proof}
    Taking a $\Q$-factorial generalised dlt model $f\colon Y\to X$ of $(X, B + M)$ (see 
    \cite[Lemma~4.5]{B-Zhang}), we have that
    \[ K_Y + B_Y + M_Y = f^*(K_X + B + M), \]
    where $B_Y$ is the sum of $f^{-1}_* B$ and the reduced $f$-exceptional divisors,
    and $(Y, B_Y + M_Y)$ is a $\Q$-factorial generalised dlt generalised pair.
    In particular, $(Y, 0)$ is $\Q$-factorial and klt. 
    Fix a small rational number $0< \epsilon \ll 1$.
    Let $g\colon W\to Y$ be a birational morphism extracting the divisors whose 
    log discrepancies with respect to $(Y, 0)$ are $\le \epsilon$.
    Then $W$ is $\epsilon$-lc.  Write
    \[ K_W + E = g^* K_Y. \]
    Then the coefficients of $E$ are $\ge 1-\epsilon$.  
    Now we can write
    \[ (f\circ g)^* L - K_W = E + g^*(B_Y + M_Y), \]
    which is clearly pseudo-effective.  Since $(f\circ g)^* L$
    is also a nef and big Cartier divisor, 
    we can apply \cite[Theorem~1.1]{Birkar-moduli-polarized} to conclude.
\end{proof}


\subsection{Relative degree}

Let $f\colon X\to Z$ be a \emph{surjective} projective morphism of quasi-projective varieties, 
and let $A$ be a $\Q$-Cartier divisor on $X$. 
For a Weil divisor $D$ on $X$, we define the \emph{relative degree} of $D$ over $Z$ with respect to $A$ 
as $\deg_{A/Z}D\coloneqq(D|_F)\cdot (A|_F)^{n-1}$, where $F$ is a general fibre of $f$ and $n=\dim F$. 
It is clear that this is a generic property, so the vertical$/Z$ irreducible components of $D$ do not contribute 
to the relative degree.  Note that $F$ may not be irreducible: 
by a general fibre, we mean the fibre over a general closed point of $Z$.
In practice, we take $A$ to be ample over $Z$; see \cite[\S 2.2]{birkar2025toroidaltoricmodelsfibrations}.


\subsection{Relatively bounded families of couples}\label{r-bnd-def}

We define (generically) relatively bounded families of couples (respectively, varieties) 
as in \cite[\S 3]{birkar2025toroidaltoricmodelsfibrations}.
Let $\mathcal{P}$ be a family of couples $(X/Z, D)$. We say $\mathcal{P}$ is \emph{generically relatively bounded} if 
there exist natural numbers $d, r$ 
such that for each $(X/Z,D)\in \mathcal{P}$, we have the following: $\dim X -\dim Z \le d$ and there is a very ample$/Z$ 
divisor $A$ on $X$ such that 
\[ \deg_{A/Z}A\le r ~~\mbox{and} ~~\deg_{A/Z}D\le r. \]
If in addition all the $(X/Z, D)\in \mc{P}$ are flat, we say that $\mc{P}$ is \emph{relatively bounded}.

When $D=0$ for every $(X/Z,D)\in \mathcal{P}$, we then refer to $\mathcal{P}$ 
as a family of generically relatively bounded (respectively, relatively bounded) varieties.

Here we collect some useful results about (generically) relatively bounded families of couples; 
for proofs, see \cite[\S 3]{birkar2025toroidaltoricmodelsfibrations}.


\begin{lem}[cf. \protect{\cite[Lemma~3.2]{birkar2025toroidaltoricmodelsfibrations}}]\label{l-bnd-couple-induced-by-fibration}
Let $W\to T$ be a projective morphism of varieties and $G$ a reduced Weil divisor on $W$. 
Let $\mathcal{P}$ be the set of couples $(Y/Z,E)$ satisfying the following: 
\begin{itemize}
	\item $Z$ is a variety equipped with a morphism $Z\to T$,
	\item $Y$ is an irreducible component of $Z\times_TW$ with reduced structure, mapping onto $Z$,
	\item the image of $Y\to W$ is not contained in $\Supp G$, and 
	\item the horizontal$/Z$ part of $E$ is contained in $\Supp (G|_Y)$, where $G|_Y$ is the divisorial part of the reduction of the  closed subscheme $Y\times_W G$ of $Y$.
\end{itemize}
Then $\mathcal{P}$ is a generically relatively bounded set of couples.  
\end{lem}

\begin{proof}
We can find effective Cartier divisors $G_1, \dots, G_r$ on $W$ so that 
\[ \Supp G = \bigcap_{i=1}^r \Supp G_i. \]
Now for any given $(Y/Z, E)$, we can apply \cite[Lemma~3.2]{birkar2025toroidaltoricmodelsfibrations} 
for some $1\le i\le r$.
\end{proof}


As for bounded set of projective couples (cf. \cite[Lemma~2.21]{B-Fano}), 
there is a universal family of 
varieties and divisors for a relatively bounded family of couples 
over smooth curves.


\begin{lem}[\protect{\cite[Lemma~3.4, Lemma~3.5]{birkar2025toroidaltoricmodelsfibrations}}]\label{l-univ-family}
Let $d$ and $r$ be natural numbers. 
Let $\mc{P}$ be the set of all couples $(X/Z, D)$ such that
\begin{itemize}
	\item [\emph{(a)}] $(X,D)$ is a couple with $\dim X = d$,
	\item [\emph{(b)}] $f\colon X\to Z$ is a projective morphism onto a smooth curve,
	\item [\emph{(c)}] every irreducible component of $D$ is horizontal over $Z$, and
	\item [\emph{(d)}] $A$ is a very ample$/Z$ divisor on $X$ such that $\deg_{A/Z}A\le r$ and $\deg_{A/Z}D\le r$.
\end{itemize}
Then there exist finitely many projective morphisms ${V}_i\to T_i$ of varieties 
and reduced divisors $C_i$ on $V_i$ such that for each $(X/Z, D)\in \mc{P}$ 
and each closed point $z\in Z$, after shrinking $Z$ around $z$ if necessary,  
there exist an $i$ and a morphism $Z\to T_i$ such that $X=Z\times_{T_i}V_i$ and $D=Z\times_{T_i}C_i$.
\end{lem}


\subsection{Irrationality}\label{defn-irr}

Given a projective variety $X$, 
we define the \emph{degree of irrationality} of $X$ following \cite{irra} as
\[\irr (X) \coloneqq \min \bigg\{ \delta>0\,\,\bigg | \,\, \begin{array}{c}
\exists \text{ degree } \delta \text{ rational dominant map} \\ X\dashrightarrow \PP^n \text{ with }\dim X =n
\end{array} \bigg\}.\]
To save space, we also say that the \emph{irrationality of $X$} is $\irr (X)$.


\begin{lem}\label{irr-bnd-above}
    Let $\mc{Q}$ be a birationally bounded family of projective varieties.
    Then the irrationality $\irr (X)$ for $X$ in $\mc{Q}$ is bounded from above.	
\end{lem}

\begin{proof}
	By assumptions, we can assume that there is a single universal family $V\to T$
	for $\mc{Q}$ as in \S \ref{b-bnd-couples}.  As irrationality is a birational invariant, 
	we can assume that the varieties in $\mc{Q}$ are closed fibres of $V\to T$.
	Embed $V/T$ into a projective space $\PP^N_T$ for some $N\in \N$.
	Taking a general projection from $\PP^N_T$ to some projective subspace of $\PP^N_T$,
	there exists a generically finite rational map $\phi\colon V\dashrightarrow \PP^n_T$ over $T$
	for some $n\le N$.  There exists a nonempty open subset $U\subseteq T$
    such that for each closed point $t\in U$, $\phi$ induces a generically finite rational map
    $\phi_t\colon V_t\dashrightarrow \PP^n$ such that $\deg \phi_t = \deg \phi$,
    where $V_t$ is the fibre of $V\to T$ over $t$.
    The complement $T\setminus U$ is a union of varieties whose dimensions are strictly less than $\dim T$.
    Then proceeding by induction on $\dim T$, we can conclude that
    the irrationality of $X\in \mc{Q}$ is bounded from above.
\end{proof}


\section{\bf Couples and toroidal geometry}\label{toroidal-geo}


\subsection{Morphisms of couples}\label{ss-tower-couples}

A \emph{morphism} $(Z,E)\to (V,C)$ between couples is a morphism $f\colon Z\to V$ of schemes such that 
$f^{-1}(C)\subseteq E$; see \cite[\S 3.6]{birkar2025toroidaltoricmodelsfibrations}.


\subsection{Toric varieties and toric morphisms}

An \emph{affine toric variety $X$ of dimension $d$} is an affine variety containing 
an algebraic torus $\mathbb{T}_X\cong \mathbb{G}_m^d$ as
a Zariski open subset such that the action of $\mathbb{T}_X$ on itself
extends to an algebraic action of $\mathbb{T}_X$ on $X$.
When the affine toric variety $X$ is normal, this is equivalent to giving a pair $(N_X, \sigma)$,
where $N_X$ is a lattice of finite rank and $\sigma$ is a strongly convex rational polyhedral cone
in $N_X\otimes_{\Z}\R$ such that $X =\Spec \mbb K [\sigma^{\vee}\cap M_X]$,
where $M_X$ is the dual lattice of $N_X$ and $\sigma^{\vee}\subset M_X\otimes_{\Z} \R$ is the dual cone of $\sigma$.
A \emph{toric morphism} of affine normal toric varieties $X$ and $Y$
is given by a linear map of lattices $\phi\colon N_X\to N_Y$
such that the image of the cone of $X$ under the extended map 
$\phi_{\R}\colon N_X\otimes_{\Z}\R \to N_Y\otimes_{\Z}\R$
is contained in the cone of $Y$.
We refer to \cite{CLS:toric} for the general theory of toric varieties.

If $D$ is the toric boundary of $X$, that is, $D$ is the reduction of the complement 
of the torus $\mathbb{T}_X$ of $X$, it is well-known that
$(X, D)$ is lc and $K_X + D\sim 0$; cf. \cite[Theorem~8.2.3]{CLS:toric}.
In this case, we say that $(X, D)$ is a \emph{toric couple}.


\subsection{Toroidal couples}\label{toroidal-couples-defn}

Let $(X,D)$ be a couple. We say the couple is \emph{toroidal} at a closed point $x\in X$ 
if there exist a \emph{normal} affine toric variety $W$ and a closed point $w\in W$ such that there is 
a $\mathbb K$-algebra isomorphism 
\[ \widehat{\mathcal{O}}_{{X},{x}}\to \widehat{\mathcal{O}}_{{W},w} \]
of completions of local rings so that the ideal of $D$ is mapped to the ideal of the toric boundary divisor 
$C\subset W$, that is, the reduction of the complement of the torus $\mathbb{T}_W$ of $W$. 
We call $\{(W,C),w\}$ a \emph{local toric model} of $\{(X,D),x\}$.
We say the couple $(X,D)$ is \emph{toroidal} if it is toroidal at every closed point,
and we call $D$ the \emph{toroidal boundary}.

In the literature, the open immersion $U_X\coloneqq X\setminus \Supp D \subset X$ is called a 
\emph{toroidal embedding}; for example, see \cite[II~\S 1]{KKMB}.
We usually denote this toroidal embedding by $(U_X\subset X)$.
Note that $U_X$ is smooth as $\mc{O}_{X,x}$ is regular if
and only if $\wh{\mc{O}}_{X,x}$ is regular.  
For citations,
we will use the notions of toroidal couples and toroidal embeddings interchangeably
to be consistent with the literature, for example, 
\cite{weak:s-stable:reduction, birkar2025toroidaltoricmodelsfibrations, KKMB}.
Moreover, if the embedding $(U_X\subset X)$, or equivalently the couple $(X, D)$,
is clear from the context, we just say that $X$ is a \emph{toroidal variety}.

Let $(X, D)$ be a toroidal couple.  If every irreducible component of $D$ is normal,
we call $(X, D)$ a \emph{strict toroidal couple}.  In this case, we say that the corresponding
toroidal embedding $(U_X\subset X)$ is a \emph{strict toroidal embedding} or
a \emph{toroidal embedding without self-intersection}; see \cite[page~57]{KKMB}.
Moreover, if the toroidal boundary $D$, or equivalently the embedding $(U_X\subset X)$, 
is clear from the context, we just say that $X$ is a \emph{strict toroidal variety}.


\begin{lem}[\protect{\cite[Lemma~3.11]{birkar2025toroidaltoricmodelsfibrations}}]\label{toroidal-normal-lc}
Let $(X,D)$ be a toroidal couple. Then $X$ is normal and Cohen-Macaulay, $K_X+D$ is Cartier, and  
$(X,D)$ is an lc pair.
\end{lem}

\begin{proof}
See \cite[\S 3]{birkar2025toroidaltoricmodelsfibrations}.
\end{proof}


\begin{prop}\label{rational-fibres}
    Let $(X, D)$ be a toroidal couple, and let $F$ be a divisor over $X$.
    Then the log discrepancy $a(F, X, D)$ is a non-negative integer;
    in particular, if $F$ is a divisor over $X$ with 
    $a(F, X, D)<1$, then $a(F, X, D)=0$.
    Moreover, if $(X, D)$ is a strict toroidal couple, and if $F$ is a divisor over $X$
    with $a(F, X, D) = 0$, then
    $F\to \centre_X F$ has irreducible and rational general closed fibres.  
\end{prop}

\begin{proof}
    If $(X, D)$ is a toroidal couple,
    by Lemma~\ref{toroidal-normal-lc}, $K_X +D$ is Cartier, and $(X,D)$ is an lc pair,
    so $a(F, X, D)$ is a non-negative integer.  
    Now we assume that $(X, D)$ is a strict toroidal couple and that $a(F, X, D)=0$. 

    \medskip
    
    \emph{Step~1}. 
    After shrinking $X$ around the generic point of $\centre_X F$, by \cite[page~195]{KKMB},
	there exist a normal affine toric variety $W$ and an \'etale morphism $p\colon X\to W$
    such that $X\times_W C$ is equal to $D$,
    where $C$ is the reduction of the complement of the torus of $W$. 

    Let $X'\to X$ be a log resolution such that $F$ is a nonsingular divisor on $X'$.
    Denote by $\eta_F$ the generic point of $F$.
	By \cite[Lemma~2.22]{Kol_singularities_of_MMP}, there exists a morphism $W_n\to W$, which is the
	composite of a sequence of blow-ups, such that 
	\begin{itemize}
		\item [(i)] there exists an induced rational map $f_n\colon X'\dashrightarrow W_n$,
                    whose domain of definition contains a neighbourhood of $\eta_F$, and
		\item [(ii)] the image $f_n(\eta_F)$ is a codimension one regular point of $W_n$.
	\end{itemize}
	By taking a further log resolution of $W_n$, we can assume that $W_n$ is also smooth.
	Taking the fibre product gives an induced rational map 
	$f_n'\colon X'\dashrightarrow X\times_W W_n$.
	We include a commutative diagram for convenience of readers as follows.
	\[\xymatrix{
	X'\ar@{-->}@/^1pc/[rrd]^{f_n}\ar[d]\ar@{-->}[rd]^{f_n'} & &  \\
    X\ar[rd]_p & X\times_W W_n\ar[r]\ar[l] &  W_n\ar[ld]  \\
	  & W & 
	}\]
	Since $p$ is \'etale, $X\times_W W_n$ is irreducible and smooth; moreover, the projection 
	$X\times_W W_n\to X$ is birational.
	Then $f_n'(\eta_F)$ is also a codimension one regular point of $X\times_W W_n$, 
	hence $f_n'$ is an isomorphism near $\eta_F$ by \cite[Lemma~3.3.24]{Liuqing}.
    Let $F_n$ be the closure of $f_n(\eta_F)$. 
    Then by \cite[\S 2.42]{Kol_singularities_of_MMP}, $a(F_n, W, C)=0$, that is,
	$F_n$ is an lc place of $(W, C)$.

    \medskip
	
	\emph{Step~2}.
	Let $\pi\colon W'\to W$ be a toric log resolution of $W$, and let 
	$\phi_n\colon W_n\dashrightarrow W'$ be
	the induced birational map.  
    If $C'$ is the reduction of the complement of the torus of $W'$, then
	$K_{W'}+C' = \pi^*(K_W+C)$.
	Thus, the closure of $\phi_n(\eta_{F_n})$ in $W'$, 
	that is, $\centre_{W'} (F_n)$, is an lc centre of $(W', C')$, where
	$\eta_{F_n}$ is the generic point of $F_n$.  Then
	$\centre_{W'}(F_n)$ is an irreducible component of the intersections of irreducible components of $C'$
	as $(W', C')$ is log smooth.
	
	By the orbit-cone correspondence of toric varieties
	(see \cite[Theorem~3.2.6, Proposition~3.2.7]{CLS:toric}), 
	each irreducible component of the intersections of 
	prime divisors of $C'$ is a disjoint
	union of finitely many orbits of the torus action,
	and each of these irreducible components is a toric variety
	$V(\tau)=\overline {O(\tau)}$
	for some orbit $O(\tau)$, where $\tau$ is a face of some cone in the fan defining $(W', C')$.
	Pick $V(\tau)$ so that $\centre_{W'}(F_n)=V(\tau)$.
	By blowing up $W'$ along $V(\tau)$, we get a toric morphism $b\colon B_{V(\tau)} (W')\to W'$.
	Then $(B_{V(\tau)}(W'), \Supp b^*C')$ is also log smooth and $\centre_{B_{V(\tau)} (W')} (F_n)$
	is a log canonical centre of the pair $(B_{V(\tau)}(W'), \Supp b^*C')$.
	We can blow up $\centre_{B_{V(\tau)}(W')} (F_n)$ further, and inductively the centre of $F_n$
	on some blow-up $\wt{W}\to W'$ will be a divisor generically isomorphic to $F_n$ 
    by \cite[Lemma~2.22]{Kol_singularities_of_MMP}.
	Since each step of the blow-ups is a toric morphism, we have that $\wt{W}\to W'$ 
	is a toric morphism and $\wt{F}\coloneqq\centre_{\wt{W}} (F_n)$ is a 
	toric variety as a torus-invariant divisor on $\wt{W}$.
	
	Denote by $V$ the image of $V(\tau)$ on $W$, that is, $V=\centre_W(F_n)$,
	which is also the closure of a torus orbit on $W$.
	Denote by $(N, \Sigma)$ the fan defining the toric variety $(W, C)$, 
    where $N$ is a lattice and $\Sigma$ is a fan in $N_{\R}$; similarly,
    denote by $(\wt{N}, \wt{\Sigma})$ the fan defining $(\wt{W}, \wt{C})$,
    where $\wt{C}$ is the reduction of the complement of the torus of $\wt{W}$.
    Let $\varphi\colon (\wt{N}, \wt{\Sigma})\to (N,\Sigma)$ be the morphism of
    lattices and fans corresponding to the toric morphism $\wt{W}\to W$.
    Then by construction, $\varphi$ is the identity morphism on lattices and $\wt{\Sigma}$ is a refinement of $\Sigma$.
    In particular, $\wt{\Sigma}$ is obtained from $\Sigma$ by a sequence of
    star subdivisions; see \cite[Theorem~11.1.9, Theorem~11.2.2]{CLS:toric}
    and \cite[\S 3.3]{CLS:toric}.
    
    Denote by $\wt{\sigma}\in \wt{\Sigma}$ the ray (respectively, $\sigma\in \Sigma$ the cone) 
    such that $\wt{F} = V(\wt{\sigma}) = \overline{O(\wt{\sigma})}$ 
    (respectively, $V = V(\sigma) = \overline{O(\sigma)}$).
    Denote by $N_{\sigma}$ the sublattice of $N$ generated as a subgroup by $\sigma\cap N$,
    and similarly, $\wt{N}_{\wt{\sigma}}$ for $\wt{\sigma}\in \wt{\Sigma}$.
    Set $N(\sigma) \coloneqq N/N_{\sigma}$ and $\wt{N}(\wt{\sigma}) \coloneqq \wt{N}/ \wt{N}_{\wt{\sigma}}$.
    Then the morphism of torus orbits
    \[ \pi_{\sigma}^{\wt{\sigma}}\colon O(\wt{\sigma}) \to O(\sigma) \]
    is induced by the surjective morphism of lattices
    \[ \overline{\varphi}_{\sigma}^{\wt{\sigma}}\colon \wt{N}(\wt{\sigma}) \to N(\sigma). \]
    Thus, the fibre $F^{\wt{\sigma}}_{\sigma}$ of $\pi_{\sigma}^{\wt{\sigma}}$
    over a closed point in $O(\sigma)$ is isomorphic to the torus 
    $\mathbb{T}_{\varphi^{-1}(N_{\sigma})/\wt{N}_{\wt{\sigma}}}$;
    see \cite[page~464]{toric-morphism}.
    Note that the fibre of $\pi^{\wt{\sigma}}_{\sigma}$ over a closed point of $O(\sigma)$
    depends only on the orbit $O(\sigma)$; see \cite[Proposition~2.1.4]{toric-morphism}.

    By generic flatness, a general closed fibre of $\wt{F}\to V$ is pure-dimensional 
    of dimension $\dim \wt{F} - \dim V$, which is equal to $\dim O(\wt{\sigma}) - \dim O(\sigma)$.
    Let $\wt{F}_c$ be a general closed fibre of $\wt{F}\to V$, where we can 
    assume that $c\in O(\sigma)$.  
    As $\dim \big(\wt{F}_c\cap (\wt{F}\setminus O(\wt{\sigma}))\big) < \dim \wt{F}_c$,
    there is a one-to-one correspondence 
    between irreducible components of $\wt{F}_c$ and irreducible components of the open subset
    $O(\wt{\sigma})_c\cong F^{\wt{\sigma}}_{\sigma}\subset \wt{F}_c$.
    Therefore, a general closed fibre of $\wt{F}\to V$ is irreducible and rational.

    \medskip
    
    \emph{Step~3}.
    Since $F_n\dashrightarrow \wt{F}$ is a birational map of nonsingular varieties over $V$,
    a general closed fibre of $F_n\to V$ is also irreducible and rational.
    Denote by $F'$ the closure of $f_n'(\eta_F)$, which is a prime divisor on $X\times_W W_n$; see Step~1.
    View $(\centre_X F)\times_W F_n$ as a closed subscheme of $X\times_W W_n$.
    As a general closed fibre of $F_n\to V$ is irreducible, $F'$ is the only irreducible component
    of $(\centre_X F)\times_W F_n$ that dominates $\centre_X F$.  
    Thus, a general closed fibre of $F'\to \centre_X F$ is isomorphic to a general closed fibre of $F_n\to V$.
    As $F\dashrightarrow F'$ is a birational map over $\centre_X F$, 
    by generic flatness, we can conclude that a general closed
    fibre of $F\to \centre_X F$ is irreducible and rational.
\end{proof}


\subsection{Toroidal morphisms}\label{toroidal-morphisms}

Now let $(X,D)$ and $(Y,E)$ be couples, and let $f\colon X\to Y$ be a morphism of varieties. 
Let $x\in X$ be a closed point 
and $y=f(x)$. We say $(X,D)\to (Y,E)$ 
is a \emph{toroidal morphism at $x$} if there exist local 
toric models $\{(W,C),w \}$ and $\{(V,B),v\}$ of 
$\{(X,D), x\}$ and $\{(Y,E), y\}$ respectively, and a toric morphism $g\colon W\to V$ of normal affine toric varieties 
so that we have a commutative diagram 
\[\xymatrix{
\widehat{\mathcal{O}}_{{X},{x}}\ar[r] & \widehat{\mathcal{O}}_{{W},w}\\
\widehat{\mathcal{O}}_{{Y},{y}} \ar[u] \ar[r] & \widehat{\mathcal{O}}_{{V},v} \ar[u]
}\]
where the vertical maps are induced by the given morphisms $f$ and $g$, and the horizontal maps are 
isomorphisms induced by the local toric models. We say the morphism
$f\colon (X, D) \to (Y, E)$ is \emph{toroidal} 
if it is toroidal at every closed point of $X$. 
Equivalently, we call the corresponding morphism 
$f\colon (U_X\subset X) \to (U_Y\subset Y)$
a \emph{toroidal morphism of toroidal embeddings}; cf. \S \ref{toroidal-couples-defn}.

By taking local toric models, it is easy to see the following result.

\begin{lem}\label{add-ver-to-toroidal}
    Let $f\colon (U_X\subset X)\to (U_Y\subset Y)$ be a toroidal morphism of toroidal embeddings.
    Then $f^{-1}(Y\setminus U_Y)$ is contained in the toroidal boundary $X\setminus U_X$.
\end{lem}


\subsection{Toroidal varieties and logarithmic geometry}

As there are no other applications of logarithmic geometry in this paper except in the proof of
Theorem~\ref{sat-base-change-intro}, 
we refer the readers to \cite{gabber2018foundationsringtheory, Kato89, Kato-toric-sing, Ogus-log-geo}
for details about logarithmic geometry.
For more details regarding the relation between toroidal geometry and logarithmic geometry,
see \cite{Q-log-geometry}.

All the log-structures in this paper are defined on Zariski sites.
A \emph{log-scheme} $(X, \mc{M})$ consists of a scheme $X$ and a \emph{log-structure} $\alpha\colon \mc{M}\to \mc{O}_X$;
see \cite[\S 1]{Kato-toric-sing}.  A \emph{log-variety} is a log-scheme whose underlying scheme is a variety.
To avoid confusion, when saying \emph{logarithmically smooth log-morphisms}, 
we mean logarithmically smooth morphisms in the category of log-schemes defined in the sense of 
\cite{Kato89, Kato-toric-sing, Ogus-log-geo}.  However, when saying a log smooth morphism, 
we mean a log smooth morphism from a couple as in \S \ref{log-smooth}.
The applications of these terminologies should be clear from the context.

Let $X$ be a variety and $\iota\colon U\hookrightarrow X$ an inclusion of an open subset.
Denote by $\mc{M}_{U/X}$ the sheaf of monoids $\mc{O}_X\cap \iota_* \mc{O}_X^*$.
We call $\mc{M}_{U/X}$ the \emph{compactifying log-structure} associated to the embedding $(U\subset X)$;
cf. \cite[\S III.1.6]{Ogus-log-geo}.
The log-variety $(X, \mc{M}_{U/X})$ is \emph{logarithmically regular} 
(or \emph{log-regular} for short) if and only if 
$(U\subset X)$ is a \emph{strict} toroidal embedding; cf. \cite[(1.7)]{Kato-toric-sing}
and \cite{Q-log-geometry}.

Let $(U_X\subset X), (U_Y\subset Y)$ be \emph{strict} toroidal embeddings.
Let $(U_X\subset X)\to (U_Y\subset Y)$ be a morphism of embeddings, that is, 
a morphism $f\colon X\to Y$ of varieties such that $f(U_X) \subseteq U_Y$.
Then $(U_X\subset X)\to (U_Y\subset Y)$
is a \emph{dominant} toroidal morphism if and only if the induced log-morphism
$(X, \mc{M}_{U_X/X})\to (Y, \mc{M}_{U_Y/Y})$ of log-varieties is logarithmically smooth in the category of log-schemes;
see \cite{denef2013remarkstoroidalmorphisms, Q-log-geometry}.


\subsection{Relative toroidalisation}

First we recall a theorem proved by Abramovich and Karu in \cite{weak:s-stable:reduction}, which 
says that any projective, surjective morphism of varieties can be made 
logarithmically smooth after birational modifications.
This result is generalised to varieties over non-closed fields of characteristic zero in \cite{ADK13}.

\begin{thm}[\protect{\cite{ADK13, weak:s-stable:reduction}}]\label{ak00-thm}
	Let $f\colon X\to B$ be a projective, surjective morphism of varieties.
	Let $Z\subset X$ be a proper closed subset.  Then there exists a commutative diagram
	\[\xymatrix{
	    (U_{X'}\subset X')\ar[r]^-{m_X}\ar[d]^{f'} & X\ar[d]^{f} \\
	          (U_{B'}\subset B')\ar[r]^-{m_B} & B
	}\]
    where $X'$ and $B'$ are nonsingular varieties, and 
	$m_B$ and $m_X$ are projective birational morphisms, such that
	\begin{itemize}
        \item [\emph{(1)}] the inclusions on the left are strict toroidal embeddings;
		\item [\emph{(2)}] $f'$ is a projective toroidal morphism of toroidal embeddings;
		\item [\emph{(3)}] let $Z'=m_X^{-1}(Z)$, then $Z'$ is a simple normal crossings divisor, and $Z'\subset X'\setminus U_{X'}$;
        \item [\emph{(4)}] the restriction of the morphism $m_X$ to $U_{X'}$ is an open embedding.
	\end{itemize}
\end{thm}


We note that the method proving this theorem in \cite{ADK13, weak:s-stable:reduction} is non-canonical and even a smooth generic fibre of $X\to B$
can be modified by it.  In \cite{ATW20}, it is shown that there is a \emph{relatively canonical toroidalisation method} that keeps the toroidal locus of $f$ unchanged;
for more details, we refer the readers to \cite[\S 1]{ATW20}.
In the next result, we show that a normalised base change of a dominant toroidal morphism is 
also toroidal.  For more details and generalisation of this result to log-structures on small \'etale sites,
we refer the readers to \cite{Q-log-geometry}.


\begin{thm}\label{sat-base-change-intro}
	Let $(U_X\subset X)$, $(U_Y\subset Y)$ and $(U_Z\subset Z)$ be strict toroidal embeddings.  Let 
    \[ f\colon (U_Y\subset Y)\to (U_Z\subset Z)~~\mbox{and}~~g\colon (U_X\subset X)\to (U_Z\subset Z) \]
	be morphisms of embeddings, that is, $f,g$ are morphisms of varieties such that 
    $f(U_Y)\subseteq U_Z$ and $g(U_X)\subseteq U_Z$.  Assume that 
	$g$ is a dominant toroidal morphism of toroidal embeddings.

    Denote by $W$ the normalisation of an irreducible
	component of $Y\times_Z X$ that dominates $Y$, and by $p\colon W\to Y$ and $q\colon W\to X$
	the induced projection morphisms respectively.
    \[\xymatrix{
	(U_W\subset W)\ar[r]^-q\ar[d]^p & (U_X\subset X)\ar[d]^g \\
	(U_Y\subset Y)\ar[r]^-f & (U_Z\subset Z)
	}\]
	Let 
    \[ U_W \coloneqq p^{-1}(U_Y)\cap q^{-1}(U_X). \]
    Then $U_W$ is a nonempty open subset of $W$, and $(U_W\subset W)$ is also a strict toroidal embedding.
    Moreover, the induced morphism of embeddings
    \[ p\colon (U_W\subset W)\to (U_Y\subset Y) \]
    is a dominant toroidal morphism of toroidal embeddings.
\end{thm}

\begin{proof}
    \emph{Step~1}.
    We also denote by 
    \[ f\colon (Y, \mc{M}_{U_Y/Y})\to (Z, \mc{M}_{U_Z/Z})~~\mbox{and}~~
    g\colon (X, \mc{M}_{U_X/X}) \to (Z, \mc{M}_{U_Z/Z}) \]
    the induced log-morphism of log-regular log-varieties,
    which should not lead to any confusion from the context.  Denote by 
    \[ \mathbf{\Omega} \coloneqq (\Omega, \mc{M}_{\Omega})~~ \mbox{and}~~
    \mathbf{\Omega}^{\fs} \coloneqq (\Omega^{\fs}, \mc{M}_{\Omega^{\fs}}) \] 
    the fibre product of $(Y, \mc{M}_{U_Y/Y})$ and $(X, \mc{M}_{U_X/X})$
    along $(Z, \mc{M}_{U_Z/Z})$ in the category of log-schemes and fs log-schemes respectively.
    By \cite[III.2.1.2]{Ogus-log-geo}, the underlying scheme $\Omega$ is isomorphic to $Y\times_Z X$, 
    and $\mc{M}_{\Omega}$ is a coherent log-structure on $\Omega$.
    
    By \cite[Remark~12.2.36~(ii)]{gabber2018foundationsringtheory}, the morphism of schemes
    $\Omega^{\fs} \to \Omega$ is a finite surjective morphism.
    By \cite[Proposition~12.3.34]{gabber2018foundationsringtheory} and 
    the proof of \cite[Proposition~12.2.35]{gabber2018foundationsringtheory},
    \[ \mathbf{\Omega}^{\fs}\to \mathbf{\Omega}\]
    is log-\'etale.  As logarithmic smoothness is stable under composition and base change
    (see \cite[Proposition~12.3.24]{gabber2018foundationsringtheory}),
    the induced log-morphism of log-schemes 
    \[ \mathbf{\Omega}^{\fs}\to (Y, \mc{M}_{U_Y/Y}) \] 
    is logarithmically smooth.  
    Then by \cite[(8.2)]{Kato-toric-sing},
    the log-scheme $\mathbf{\Omega}^{\fs}$ is log-regular.
    By \cite[Corollary~12.5.29]{gabber2018foundationsringtheory}, $\Omega^{\fs}$
    is a normal scheme, hence $\Omega^{\fs}$ is a disjoint union 
    of normal varieties.  Thus, every irreducible component
    of $\Omega^{\fs}$ is a strict toroidal variety. 

    Let $U$ be the fibre product $U_Y\times_{U_Z} U_X$.
    By \cite[Lemma~6.2]{Q-log-geometry}, $U$ is nonempty.
    Let $(U, \mc{M}_U)$ be the log-scheme endowed with trivial log-structure.
    It is evident that $\mc{M}_U$ is the restriction of $\mc{M}_{\Omega}$ to $U$.
    Denote by $\Omega^{\fs,*}$ the maximal open subset of $\Omega^{\fs}$ on which $\mc{M}_{\Omega^{\fs}}$
    is trivial.
    Then by the construction in \cite[Proposition~12.2.35]{gabber2018foundationsringtheory},
    $U$ is also an open subscheme of $\Omega^{\fs}$ satisfying $\mc{M}_U = \mc{M}_{\Omega^{\fs}}|_U$, 
    whence $U\subseteq \Omega^{\fs,*}$.  Moreover, 
    by \cite[Corollary~12.3.27]{gabber2018foundationsringtheory}, 
    the scheme $U$ is nonsingular in the usual sense.

    \medskip

    \emph{Step~2}.
    Let $W'$ be an irreducible component of $Y\times_Z X$ that dominates $Y$.
    By \cite[Lemma~6.2]{Q-log-geometry}, the inverse image of $U_X$ in $W'$ is a dense open subset of $W'$.
    Let $W$ be the normalisation of $W'$, and let $p\colon W\to Y$
    and $q\colon W\to X$ be the projections. 
    Then $p^{-1}(U_Y)\cap q^{-1}(U_X)$ is a nonempty dense open subset of $W$. 
    Since $U$ is a nonsingular scheme, there is a unique irreducible component $U_{W'}$ of $U$
    contained in $W'$.
    Then the normalisation $W\to W'$ is an isomorphism over $U_{W'}$.
    Denote by $U_W$ the inverse image of $U_{W'}$ in $W$.
    Then $U_W = p^{-1}(U_Y)\cap q^{-1}(U_X)$.
    Let $W''$ be the irreducible component of $\Omega^{\fs}$ that contains $U_{W'}$.
    As $U_{W'}$ embeds into both $W'$ and $W''$,
    the finite morphism $W''\to W'$ is also birational.  
    Since $\Omega^{\fs}$ is a normal scheme, we see that $W''\cong W$.

    Denote by $W^*\subseteq W$ the restriction of $\Omega^{\fs, *}$ to $W$.
    By \cite[Remark~12.2.8]{gabber2018foundationsringtheory}, $W^*$ is mapped into $U_Y, U_X$ 
    via $p\colon W\to Y, q\colon W\to X$ respectively.  Then 
    \[ U_W\subseteq W^* \subseteq p^{-1}(U_Y)\cap q^{-1}(U_X) = U_W. \]
    Thus, $U_W = W^*$.
    Then $(U_W\subset W)$ is a strict toroidal embedding.
    Moreover, since $\mathbf{\Omega}^{\fs} \to (Y, \mc{M}_{U_Y/Y})$ is logarithmically smooth, 
    we can conclude that $p\colon (U_W\subset W)\to (U_Y\subset Y)$ is a dominant toroidal
    morphism of toroidal embeddings.
\end{proof}


The following result is crucial in the proof of Theorem~\ref{bnd-irra}.

\begin{thm}[cf. \protect{\cite[Theorem~1.2.12]{ATW20}}]\label{functorial-toroidalization}
	Let $\Phi \colon \mf{X} \to \mf{B}$ be a projective, 
	surjective morphism of varieties with geometrically integral generic fibre.  
	Let $\mf{Z}_{\mf{X}}\subset \mf{X}$ and $\mf{Z}_{\mf{B}}\subset \mf{B}$ be proper closed subsets respectively.	
	Then there exists a 
	toroidal morphism $\Phi'$ such that:
	\begin{itemize}
		\item [\emph{(i)}] \emph{[Existence]} There is a commutative diagram as follows:
	          \[\xymatrix{
	          (U_{\mf{X}'}\subset \mf{X}')\ar[r]^-{m_{\mf{X}}}\ar[d]^{\Phi'} &\mf{X}\ar[d]^{\Phi} \\
	          (U_{\mf{B}'}\subset \mf{B}')\ar[r]^-{m_{\mf{B}}} & \mf{B}
	          }\]
	          such that $m_{\mf{B}}$ and $m_{\mf{X}}$ are projective birational morphisms, 
	          the inclusions on the left are strict toroidal embeddings, 
	          $m_{\mf{X}}^{-1}(\mf{Z}_{\mf{X}})$ is contained in $\mf{X}'\setminus U_{\mf{X}'}$, 
	          $m_{\mf{B}}^{-1}(\mf{Z}_{\mf{B}})$ is contained in $\mf{B}'\setminus U_{\mf{B}'}$,
	          and $\Phi'$ is a surjective, projective toroidal morphism.
	     \item [\emph{(ii)}] \emph{[Base change functoriality]} 
	         There is a nonempty open subset $U_{\mf{B}}\subset \mf{B}\setminus \mf{Z}_{\mf{B}}$
	         satisfying the following property.
	         Let $i\colon C\to \mf{B}$ be a morphism from a smooth curve.  
	         Assume that the image $i(C)$ is not entirely contained in the 
	         closed subset $\mf{B}\setminus U_{\mf{B}}$.  
             Then $\mf{X}\times_{\mf{B}} C$ admits a unique irreducible component dominating $C$.
	         Denote by $Y$ the reduction of this irreducible component of $\mf{X}\times_{\mf{B}} C$ 
	         with projection morphism $f\colon Y\to C$.  Consider the diagram arising from 
	         normalised base change as follows:
	         \[\xymatrix{
	         (U_{Y'}\subset Y')\ar[ddd]^{f'}\ar[rrr]^{m_Y}\ar[rd] & & & Y\ar[ddd]^f\ar[ld] \\
	           & (U_{\mf{X}'}\subset \mf{X}')\ar[r]^-{m_{\mf{X}}}\ar[d]^{\Phi'} &\mf{X}\ar[d]^{\Phi} & \\
	           & (U_{\mf{B}'}\subset \mf{B}')\ar[r]^-{m_{\mf{B}}} & \mf{B} & \\
	         (U_{C'}\subset C')\ar[rrr]^{m_C}\ar[ru] & & & C\ar[lu]_i
	         }\]
	         where $C'$ (respectively, $Y'$) is 
	         the normalisation of the main component of $\mf{B}'\times_{\mf{B}} C$  
	         (respectively, of $\mf{X}'\times_{\mf{B}'} C'$), and
	         $U_{C'}$ (respectively, $U_{Y'}$) is the inverse image of $U_{\mf{B}'}$ in $C'$
	         (respectively, of $U_{\mf{X}'}$ in $Y'$).
	         Then $U_{C'}, U_{Y'}$ are nonempty open subsets of $C', Y'$ respectively,
             $f'$ is also a toroidal morphism of strict toroidal embeddings, and all the properties 
	         in \emph{(i)} are satisfied.
	         In particular, $m_C$ and $m_Y$ are projective birational morphisms, and
	         the inverse image of $\mf{Z}_{\mf{B}}$ in $C'$ (respectively, of 
	         $\mf{Z}_{\mf{X}}$ in $Y'$) is contained in $C'\setminus U_{C'}$ (respectively, in $Y'\setminus U_{Y'}$).
	\end{itemize}
\end{thm}

\begin{proof}
	\emph{Step~1}.
	To include the closed subset $\mf{Z}_{\mf{X}}$ into the toroidal structure, we first 
	modify $\Phi\colon \mf{X}\to \mf{B}$ by Theorem~\ref{ak00-thm} as in the following commutative diagram
	\[\xymatrix{
	(U_{\mf{X}^{\circ}}\subset \mf{X}^{\circ})\ar[r]^-{m_{\mf{X}^{\circ}}}\ar[d]^{\Phi^{\circ}} & \mf{X}\ar[d]^{\Phi} \\
	(U_{\mf{B}^{\circ}}\subset \mf{B}^{\circ})\ar[r]^-{m_{\mf{B}^{\circ}}} & \mf{B}
	}\]
	where $m_{\mf{X}^{\circ}}$ and $m_{\mf{B}^{\circ}}$ are projective birational morphisms,
	the inclusions on the left are strict toroidal embeddings, $m_{\mf{X}^{\circ}}^{-1}(\mf{Z}_{\mf{X}})$
	is contained in the toroidal boundary $\mf{X}^{\circ}\setminus U_{\mf{X}^{\circ}}$,
	and $\Phi^{\circ}$ is a projective, surjective toroidal morphism of strict toroidal embeddings.

    \medskip
	
	\emph{Step~2}. 
	Now we include the closed subset $\mf{Z}_{\mf{B}}$ into the toroidal boundary by an additional 
	modification on $\mf{B}^{\circ}$.  Consider the closed subset 
	$\mf{Z}^{\circ}\subset \mf{B}^{\circ}$ that is
	the union of the inverse image of 
	$\mf{Z}_{\mf{B}}$ in $\mf{B}^{\circ}$ and the toroidal boundary 
	$\mf{B}^{\circ}\setminus U_{\mf{B}^{\circ}}$.
	Let $m_{\mf{B}'}\colon\mf{B}'\to \mf{B}^{\circ}$ be a log resolution of 
	$(\mf{B}^{\circ}, \mf{Z}^{\circ})$
	such that $\mf{B}'$ is smooth and the inverse image 
	of $\mf{Z}^{\circ}$ in $\mf{B}'$ is a simple normal crossing divisor.
	Set $U_{\mf{B}'}$ as the open complement $\mf{B}'\setminus m_{\mf{B}'}^{-1}(\mf{Z}^{\circ})$.
	Then $(U_{\mf{B}'}\subset \mf{B}')$ is a strict toroidal embedding and 
	\[m_{\mf{B}'}\colon (U_{\mf{B}'}\subset \mf{B}') \to (U_{\mf{B}^{\circ}}\subset \mf{B}^{\circ}) \]
    is a morphism of embeddings.
    
	Taking the normalised base change $\Phi'$ of the toroidal morphism $\Phi^{\circ}$
	(see Theorem~\ref{sat-base-change-intro}), we have
	a commutative diagram
	\[\xymatrix{
	(U_{\mf{X}'}\subset \mf{X}')\ar[r]^-{m_{\mf{X}'}}\ar[d]^{\Phi'} & (U_{\mf{X}^{\circ}}\subset \mf{X}^{\circ})\ar[r]^-{m_{\mf{X}^{\circ}}}\ar[d]^{\Phi^{\circ}} & \mf{X}\ar[d]^{\Phi} \\
	(U_{\mf{B}'}\subset \mf{B}')\ar[r]^-{m_{\mf{B}'}} & (U_{\mf{B}^{\circ}}\subset \mf{B}^{\circ})\ar[r]^-{m_{\mf{B}^{\circ}}} & \mf{B}
	}\]
	where $\mf{X}'$ is the normalisation of the main component of $\mf{B}'\times_{\mf{B}^{\circ}} \mf{X}^{\circ}$,
    $m_{\mf{X}'}$ and $m_{\mf{B}'}$ are projective birational morphisms, 
	$(U_{\mf{X}'}\subset \mf{X}')$ is also a strict toroidal embedding, the open subset 
	$U_{\mf{X}'} \subset \mf{X}'$ is the intersection of the open subsets $(\Phi')^{-1}(U_{\mf{B}'})$ and $m_{\mf{X}'}^{-1}(U_{\mf{X}^{\circ}})$,
    and $\Phi'$ is a surjective, projective toroidal morphism.
	
	Denote by $m_{\mf{X}}$ and $m_{\mf{B}}$ the composite projective birational morphisms
	\[ m_{\mf{X}^{\circ}}\circ m_{\mf{X}'}~~ \mbox{and}~~ m_{\mf{B}^{\circ}} \circ m_{\mf{B}'}.\]
	By construction, $m_{\mf{X}}^{-1}(\mf{Z}_{\mf{X}})$ (respectively, $m_{\mf{B}}^{-1}(\mf{Z}_{\mf{B}})$) 
	is contained in $\mf{X}'\setminus U_{\mf{X}'}$ (respectively, in $\mf{B}'\setminus U_{\mf{B}'}$).  This proves (i).

    \medskip
	
	\emph{Step~3}.
	Let $V_{\mf{B}}\subset \mf{B}$ be the smallest closed subset such that
    \begin{itemize}
        \item $V_{\mf{B}}$ contains the image of the toroidal boundary $\mf{B}'\setminus U_{\mf{B}'}$,
        \item $m_{\mf{B}}$ is an isomorphism over $\mf{B}\setminus V_{\mf{B}}$, 
        \item every closed fibre of $\Phi$ over $\mf{B}\setminus V_{\mf{B}}$ is integral, and
        \item for every closed point $b\in \mf{B}\setminus V_{\mf{B}}$, let $b'$ be its inverse image in $\mf{B}'$, then $m_{\mf{X}}$ induces a birational morphism $\mf{X}'_{b'}\to \mf{X}_b$ on the fibres 
        (in particular, $\mf{X}'_{b'}$ is irreducible).
    \end{itemize}
	It is evident that $U_{\mf{B}} \coloneqq \mf{B}\setminus V_{\mf{B}}$ is a nonempty open subset of $\mf{B}$.
	Moreover, by construction, $U_{\mf{B}}$ is contained in $\mf{B}\setminus \mf{Z}_{\mf{B}}$
	as the closed subset $\mf{Z}_{\mf{B}}$ is contained in $V_{\mf{B}}$.

    Let $i\colon C\to \mf{B}$ be a morphism from a smooth curve such that the image $i(C)$
    is not entirely contained in $V_{\mf{B}}$.  As $m_{\mf{B}}$ is an isomorphism over $U_{\mf{B}}$,
    $\mf{B}'\times_{\mf{B}} C$ admits a unique irreducible component that is birational to $C$.
    Denote by $C'$ the normalisation of this irreducible component, then the induced morphism
    $C'\to C$ is an isomorphism of nonsingular curves.  Let $U_{C'}$ be the inverse image of $U_{\mf{B}'}$
    in $C'$, which is nonempty as $V_{\mf{B}}$ contains the image of $\mf{B}'\setminus U_{\mf{B}'}$.  
    It is evident that $(U_{C'}\subset C')$ is a strict toroidal embedding.
    Note that $U_{C'}$ may be equal to the whole $C'$.
    
    By the construction of $U_{\mf{B}}$, a general closed fibre of $\mf{X}\times_{\mf{B}} C\to C$ is irreducible,
    hence there is a unique irreducible component $X$ of $\mf{X}\times_{\mf{B}} C$
    that dominates $C$.  For the same reason, $\mf{X}'\times_{\mf{B}'} C'$ admits a unique
    irreducible component $X'$ that dominates $C'$.  Moreover, $m_{\mf{X}}$ induces 
    a birational morphism $X'\to X$.
    By \cite[Chap.~I, (5.1.5)]{EGA-I}, there is an induced birational morphism 
    $X'_{\red} \to X_{\red}$.  Let $Y\coloneqq X_{\red}$, and let $Y'$ be the normalisation of $X'_{\red}$.
    Then it is evident that the induced morphism $m_Y\colon Y'\to Y$ is projective and birational.

    Let $f'$ be the morphism $Y'\to C'$.
    Let $U_{Y'}$ be the inverse image of $U_{\mf{X}'}$ in $Y'$.
    Note that $U_{Y'}\cap (f')^{-1}(U_{C'})$ is equal to $U_{Y'}$ as the inverse image of $U_{\mf{B}'}$
    in $\mf{X'}$ contains $U_{\mf{X}'}$; see Lemma~\ref{add-ver-to-toroidal}.
    Applying Theorem~\ref{sat-base-change-intro} to the toroidal morphism $\Phi'$
    and the morphism of embeddings $(U_{C'}\subset C')\to (U_{\mf{B}'}\subset \mf{B}')$, 
    we see that $U_{Y'}$ is nonempty, $(U_{Y'}\subset Y')$ is also a strict toroidal embedding,
    and the induced morphism of embeddings $f'\colon (U_{Y'}\subset Y')\to (U_{C'}\subset C')$
	is a toroidal morphism of strict toroidal embeddings.  This concludes (ii).
\end{proof}


\section{\bf Bounding irrationality of degenerations}\label{pf-bnd-irra}

In this section, we prove Theorem~\ref{bnd-irra} about the generic rational fibration structure 
of irreducible fibres on degenerations of klt Fano fibrations.
First we recall the setup of this result.
Fix $d\in \N$, $\epsilon > 0$ and $t\in (0, 1]$.
Denote by $\mc{F}_{d, \epsilon, t}$ the set of data $(X/Z, tF)$, where
\begin{itemize}
	\item $Z$ is a smooth curve,
	\item $f\colon X\to Z$ is a klt Fano fibration with $\dim X =d$, 
          that is, $X$ is a klt variety of dimension $d$, $f$ is a contraction, and $-K_X$ is ample$/Z$,
    \item $X$ is $\epsilon$-lc over the generic point $\eta_Z$ of $Z$,
    \item $F$ is the reduction of an irreducible fibre of $f$ over a closed point $z\in Z$, and
    \item $(X, tF)$ is log canonical.
\end{itemize}

By decreasing $t$ if necessary, we can assume that $t\in (0,1]$ is a rational number.
Then by the boundedness of complements (see \cite[Theorem~1.8]{B-Fano}), 
there exists an $n\in \N$ depending only on $d$ and $t$ such that
there exists a monotonic $n$-complement $K_X + B^+$ of $K_X + t F$ over $z\in Z$.
Up to shrinking $Z$ around $z$ if necessary, 
we can assume that the boundary $B^+$ is defined over the whole $Z$.
More precisely, there exists an effective $\Q$-divisor $B^+$ on $X$ such that
\begin{itemize}
    \item $(X, B^+)$ is log canonical,
	\item $n(K_X + B^+) \sim 0 /Z$,
	\item $t F\le B^+$, and
	\item $a(F, X, B^+) < 1$.
\end{itemize}
Note that the last condition on log discrepancy follows from $tF\le B^+$.
Indeed, since $F$ is a divisor on $X$, the log discrepancy $a(F, X, tF)$ is equal to $1-t$, which is $<1$;
moreover, as $tF\le B^+$, we get $a(F, X, B^+)\le a(F, X, tF) < 1$.


\subsection{Modification to generically smooth couples} \label{mdf-to-g-sm-complements}

We first show that we can modify the fibration $(X/Z, tF)$ to a relatively bounded fibration 
whose general fibres are smooth.

\begin{lem}\label{bir-to-rbnd-couples}
    There is a relatively bounded family $\mc{P}$ of couples $(Y/Z, D)$ such that for each 
    $(X/Z, tF)\in \mc{F}_{d, \epsilon, t}$ with an $n$-complement $K_X + B^+$ as above,
    we have a couple $(Y/Z, D)\in \mc{P}$
    admitting a birational map $\phi\colon X\dashrightarrow Y$ over $Z$ such that
    \begin{itemize}
    	\item [\emph{(1)}] $\phi$ is an isomorphism over the generic point of $Z$, and
    	\item [\emph{(2)}] $\Supp B^+$ is mapped isomorphically to $\Supp D$ over the generic point of $Z$.
    \end{itemize}
\end{lem}

\begin{proof}
    For each $(X/Z, tF)\in \mc{F}_{d, \epsilon, t}$, by abusing the notation slightly, 
    we also write this data as $f\in \mc{F}_{d, \epsilon, t}$, where $f$ denotes the contraction $X\to Z$.

    \medskip
    
    \emph{Step~1}.
	For each $f\colon X\to Z$ in $\mc{F}_{d, \epsilon, t}$, denote by $U_f$ 
    the maximal open subset of $Z$ over which every closed fibre of $f$ is an $\epsilon$-lc Fano variety
    of dimension $d-1$.
    Denote by $\mc{F}_{d, \epsilon, t}^{\text{fb}}$ the set of 
    all closed fibres of all $f\in \mc{F}_{d, \epsilon, t}$ over $U_f$.
    Denote by $\mc{F}_{d, \epsilon, t}^{\circ}$ the set of all morphisms that are the restrictions
    of $f\in \mc{F}_{d, \epsilon, t}$ over $U_f$.
    That is, every member in $\mc{F}_{d, \epsilon, t}^{\circ}$ is equal to $f|_{f^{-1}(U_f)}$
    for some $f\in \mc{F}_{d, \epsilon, t}$.

    We claim there is a fixed $\ell\in \N$ depending only on $d, \epsilon$ such that 
    for every $f\in \mc{F}_{d, \epsilon, t}$, there exists a nonempty open subset $U_f'\subseteq U_f$
	satisfying that $-\ell K_X$ is very ample$/U_f'$.
	By boundedness of $\epsilon$-lc Fano varieties (see \cite[Theorem~1.1]{B-BAB}), 
	we can assume that $W\to R$ is the universal family of $\mc{F}_{d, \epsilon, t}^{\text{fb}}$,
	where $W$ and $R$ are quasi-projective varieties.
    Denote by $R'\subset R$ an open subset such that
    \begin{itemize}
        \item $W'\to R'$ is flat, where $W'\coloneqq R'\times_R W$,
        \item $K_{W'/R'}$ is well-defined, and
        \item for every closed point $p\in R'$, if $W_p$ is the fibre of $W\to R$ over $p$, then $K_{W_p}$ is the restriction of $K_{W'/R'}$ to $W_p$.
    \end{itemize} 
    As every member in $\mc{F}_{d, \epsilon, t}^{\text{fb}}$ is a Fano variety, 
	$-K_{W'/R'}$ is ample$/R'$.  Then there exists a fixed $\ell\in \N$ such that 
	$-\ell K_{W'/R'}$ is very ample$/R'$.
	Let $\mc{F}_{d, \epsilon, t}'\subseteq \mc{F}_{d, \epsilon, t}$ be the subset of contractions $f\colon X\to Z$
	such that there are only finitely many closed points of $U_f$ over which the fibres
    of $f$ are closed fibres of $W'\to R'$.  
    Then for every contraction $f\colon X\to Z$ in $\mc{F}_{d, \epsilon, t}\setminus \mc{F}_{d, \epsilon, t}'$,
	there are infinitely many closed points $\Set{z_j}_j$ in $U_f$
	such that every fibre $X_{z_j}$ of $f$ over $z_j$ is a closed fibre of $W'\to R'$.
	Since $\dim (R\setminus R')<\dim R$, by induction on $\dim R$, 
    we can assume that $\mc{F}_{d, \epsilon, t}' = \emptyset$.
    
	Denote by $Z'$ the subset of $Z$ such that $(-\ell K_X)|_{X_z}$ 
    is very ample for every $z\in Z'$.
	By \cite[Proposition~(9.6.3)]{EGA-IV-3}, $Z'$ is a constructible subset of $Z$.
    Moreover, by assumption, $Z'$ is an infinite subset of $Z$, then we
    can conclude that $Z'$ contains a dense open subset of $Z$.
	In particular, $(-\ell K_X)|_{X_{\eta_Z}}$ is very ample over $\Spec K(Z)$,
    where $\eta_Z$ is the generic point of $Z$ and $K(Z)$ is the function field of $Z$.
    Then $(-\ell K_X)|_{X_{\eta_Z}}$ induces a closed immersion 
    $\iota\colon X_{\eta_Z}\to \PP^N_{K(Z)}$ for some $N\in \N$.
    By spreading out $\iota$, there is a morphism 
    \[ \iota_{Z''}\colon X\times_Z Z'' \to \PP^N_{Z''} \] 
    over an open subset $Z''\subset U_f$ that extends $\iota$.
    Up to shrinking $Z''$, we can assume that $\iota_{Z''}$ is also a closed immersion,
    and that $\mc{O}_X (-\ell K_X)$ is isomorphic to $\iota_{Z''}^* \mc{O}_{\PP^N_{Z''}}(1)$ over $Z''$
    by \cite[7.27, 12.93]{GW-agbook}.
    Thus, $-\ell K_X$ is very ample over $Z''$.
    Then it suffices to take $U_f' = Z''$.  This concludes the claim.

    Now we replace the open subset $U_f\subseteq Z$ by $U_f'$.
	Then by the boundedness of $\mc{F}_{d, \epsilon, t}^{\text{fb}}$, 
	the relative degree of $-\ell K_X$ over $U_f$ is bounded from above.
	Thus, $\mc{F}_{d, \epsilon, t}^{\circ}$ is relatively bounded.

    \medskip
	
	\emph{Step~2}.
	For every $f\in \mc{F}_{d, \epsilon, t}$, let $A$ be the divisor $-\ell K_X$ which is 
	very ample$/U_f$.  Note that if $X_s$ is a general closed fibre of $X\to Z$, then
    the restriction of $K_X$ to $X_s$ is equal to $K_{X_s}$.
    Thus, by the relative boundedness of $\mc{F}_{d, \epsilon, t}^{\circ}$,
    we have that $(A|_{X_s})^{d-2}\cdot (-n K_{X_s})\le r$, where
    $r$ is a natural number depending only on 
    $\mc{F}_{d, \epsilon, t}^{\circ}$.  Denote by $C$ the reduced divisor supported 
    on the horizontal$/Z$ part of $B^+$.  Since $nB^+$ is an integral effective divisor, 
    $nB^+ - C$ is also effective, and
    \[ C + (n B^+ - C) \sim -nK_X/Z. \]
    Thus, on a general closed fibre $X_s$ of $X\to Z$, we have 
    \[ \deg_{A/Z} C =  (A|_{X_s})^{d-2}\cdot (C|_{X_s})\le (A|_{X_s})^{d-2}\cdot ((-nK_X)|_{X_s})
    = (A|_{X_s})^{d-2}\cdot (- n K_{X_s})\le r. \]
    Denote by $\mc{P}^{\circ}$ the family of all the couples 
    $(f^{-1}(U_f)/U_f, C|_{f^{-1}(U_f)})$.  
    Then $\mc{P}^{\circ}$ is a relatively bounded set of couples.
    By Lemma~\ref{l-univ-family}, there are finitely many projective morphisms 
	of varieties $V_i\to T_i$ and reduced divisors $C_i$ on $V_i$ 
	such that each couple in $\mc{P}^{\circ}$
	comes from a base change of some $(V_i, C_i)\to T_i$ (after shrinking
	the base $U_f$ if necessary).  Then for every $f\colon X\to Z$
	in $\mc{F}_{d, \epsilon, t}$, we can assume that there 
	is a morphism from $U_f$ to some $T_i$
	such that $f^{-1}(U_f)$ (respectively, $C|_{f^{-1}(U_f)}$)
	is equal to the fibre product $U_f\times_{T_i} V_i$
	(respectively, $U_f\times_{T_i} C_i$).

    \medskip
	
	\emph{Step~3}.
	For every $T_i$, let $\wt{T}_i$ be a projective compactification of $T_i$.  
	Embed $V_i\to T_i$ into some $\PP^m\times T_i$; 
	denote by $\wt{V}_i$ (respectively, by $\wt{C}_i$) the reduced scheme-theoretic closure of 
	$V_i$ (respectively, of $C_i$) in $\PP^m\times \wt{T}_i$.  
	Then $(\wt{V}_i, \wt{C}_i)\to \wt{T}_i$ is a projective compactification
	of $(V_i, C_i)\to T_i$, where $\wt{V}_i$, $\wt{C}_i$, and $\wt{T}_i$ are all projective.  
	As $Z$ is a smooth curve, the morphism $U_f \to T_i$ extends to a morphism $Z\to \wt{T}_i$.
	Denote by $\mc{P}$ the set of all couples $(Y/Z, D)$
	constructed from the
	base changes $Z\times_{\wt{T}_i} \wt{V}_i$ and $Z\times_{\wt{T}_i} \wt{C}_i$ 
	for morphisms $Z\to \wt{T}_i$ 
	as in Lemma~\ref{l-bnd-couple-induced-by-fibration}, that is,
    $Y$ is the reduction of the main component of $Z\times_{\wt{T}_i} \wt{V}_i$, and
    $D$ is the reduction of the horizontal$/Z$ divisorial part of $Z\times_{\wt{T}_i} \wt{C}_i$.
    Then $\mc{P}$ is relatively bounded.  
    Also, by construction, every $f\colon X\to Z$ in $\mc{F}_{d, \epsilon, t}$ admits
	a birational map $\phi\colon X\dashrightarrow Y$ over $Z$ to some $(Y/Z, D)$ in $\mc{P}$;
	moreover, $(X/Z, \Supp B^+)$ and $(Y/Z, \Supp D)$ have the same general closed fibres.
\end{proof}


\begin{lem}\label{bir-to-rlbnd-generic-sm-couples}
    There exists a relatively bounded family $\mc{F}_{d, \epsilon, t}^{\text{sm}}$ of couples $(Y/Z, D)$
    such that 
    \begin{itemize}
    	\item [\emph{(1)}] $(Y, D)\to Z$ is generically log smooth (see \S \ref{log-smooth}), 
       	\item [\emph{(2)}] every $(X/Z, tF)$ in $\mc{F}_{d, \epsilon, t}$ admits a birational map 
       	                   $\phi\colon X\dashrightarrow Y$ over $Z$ to some $(Y/Z, D)$ in $\mc{F}_{d, \epsilon, t}^{\text{sm}}$, and
    	\item [\emph{(3)}] over the generic point of $Z$, we have the following:
    	      \begin{itemize}
    	      \item [\emph{(i)}] $\phi$ does not contract any divisor, and
    	      \item [\emph{(ii)}] the divisor $D$ is supported on the union of $\Supp \phi_* B^+$ and the 
    	                          supports of all horizontal$/Z$ exceptional divisors of $\phi^{-1}$, where $K_X + B^+$
                                  is the complement as in Lemma~\ref{bir-to-rbnd-couples}.
    	      \end{itemize}
    \end{itemize}
\end{lem}

\begin{proof}
    Take the relatively bounded family of couples $\mc{P}$ in Lemma~\ref{bir-to-rbnd-couples}.
	Then there are finitely many universal projective morphisms 
	of varieties and reduced divisors as in Lemma~\ref{l-univ-family}.   
	We can assume that there is only one such universal projective morphism $(V, C)\to T$ for $\mc{P}$
	with $V$ and $T$ quasi-projective varieties and $C$ a reduced horizontal$/T$ divisor on $V$.  
	Take a log resolution $V^{\text{sm}}\to V$ of the couple $(V,C)$.  
	Denote by $C^{\text{sm}}$ the reduced divisor supported on the union of the birational 
	transform of $C$ and all horizontal$/T$ exceptional divisors of $V^{\text{sm}}\to V$.
	
	By generic smoothness, there exists an open
	dense subset $U_1\subset T$ such that $(V^{\text{sm}}, C^{\text{sm}})$ is log smooth over $U_1$.
	Moreover, we can assume that the images of all vertical$/T$ exceptional divisors of $V^{\text{sm}}\to V$
	do not intersect $U_1$.
	Shrinking $Z$ around the closed point $z\in Z$ if necessary, 
    there is a morphism $Z\to T$ as in Lemma~\ref{l-univ-family}.
	If the image of $Z\to T$ is not entirely contained in $T_1 \coloneqq T\setminus U_1$, 
	we can take the couple $(Y/Z, D)$ as the reduction of the main component of $V^{\text{sm}} \times_T Z$ 
	equipped with the reduced divisor which is the reduction 
	of the horizontal$/Z$ divisorial part of $C^{\text{sm}}\times_T Z$. 
	Write $T_1$ as a union of finitely many varieties. 
    Then the result follows by induction on $\dim T$.
\end{proof}


\subsection{Proof of bounded irrationality of irreducible fibres}\label{main-proof}

\begin{thm}[=Theorem~\ref{bnd-irra}]
    Fix positive real numbers $\epsilon> 0$, $t\in (0, 1]$ and a natural number $d$.
    Assume that $f\colon X\to Z$ is a klt Fano fibration with $\dim X =d$ such that
    \begin{itemize}
    	\item [\emph{(1)}] $Z$ is a smooth curve,
    	\item [\emph{(2)}] $X$ is $\epsilon$-lc over the generic point of $Z$, and
    	\item [\emph{(3)}] $F$ is the reduction of an irreducible closed fibre of $f$ and $(X, t F)$ is lc.
    \end{itemize}
    Then there is a dominant rational map $F \dashrightarrow C$ whose general fibres are irreducible and rational,
    and $C$ is a bounded smooth projective variety depending only on $d, \epsilon, t$,
    hence with bounded degree of irrationality.
\end{thm}

\begin{proof}
By decreasing $t$ if necessary, we can assume that $t$ is a rational number.
Then we construct the $n$-complement $K_X + B^+$ and the family of 
couples $\mc{F}_{d, \epsilon, t}^{\text{sm}}$
as in \S \ref{mdf-to-g-sm-complements}, where $n\in \N$ depends only on $d,t$.

\medskip

\emph{Step~1}.
Recall that the family of couples $\mc{F}_{d, \epsilon, t}^{\text{sm}}$ 
constructed in Lemma~\ref{bir-to-rlbnd-generic-sm-couples}
is relatively bounded.
Denote by $(\mf{X}_i, \mf{D}_i)\to \mf{B}_i$ 
the collection of universal families of couples of 
$\mc{F}_{d, \epsilon, t}^{\text{sm}}$ as in Lemma~\ref{l-univ-family}.
We can assume that there is only one such projective morphism 
\[ \Phi\colon (\mf{X}, \mf{D})\to \mf{B}. \]
By the construction in the proof of Lemma~\ref{bir-to-rlbnd-generic-sm-couples},
there exists a maximal open subset $\mf{B}^{\circ}\subset \mf{B}$ 
such that $(\mf{X}, \mf{D})$ is log smooth over $\mf{B}^{\circ}$, and that
every fibre of $\Phi$ over a closed point of $\mf{B}^{\circ}$ is a smooth projective variety.
Moreover, after shrinking $Z$ around the closed point $z\in Z$ if necessary,
we can assume that every $(X/Z, tF)\in \mc{F}_{d, \epsilon, t}$ admits a morphism 
$Z\to \mf{B}$ whose image intersects $\mf{B}^{\circ}$. 
That is, for every $(X/Z, tF)\in \mc{F}_{d, \epsilon, t}$, there is
a commutative diagram for some $(Y/Z, D)\in \mc{F}_{d, \epsilon, t}^{\text{sm}}$ as follows,
\[\xymatrix{
X\ar[rd]\ar@{-->}[r]^{\phi} & Y\ar[d]\ar[r] & (\mf{X}, \mf{D})\ar[d]^{\Phi} \\
 & Z\ar[r] & \mf{B} 
}\]
where $\phi\colon X\dashrightarrow Y$ is a birational map over $Z$.
If the closure of the image of $F$ in $\mf{X}$ is an irreducible component 
of a closed fibre of $\Phi$, then $F$ is birationally bounded, 
hence $F$ has bounded irrationality by Lemma~\ref{irr-bnd-above}.
In this case, we take $C$ to be a resolution of the closure of the image of $F$ in $\mf{X}$.
Note that given a bounded family of varieties there is a bounded family of resolutions for
such varieties, hence we can take $C$ to be a bounded smooth projective variety.
So, in the rest of the proof, we can assume that the closure of the image of $F$ in $\mf{X}$ is not an
irreducible component of a closed fibre of $\Phi$.

Assume that the image of $z\in Z$ in $\mf{B}$ is a closed point of $\mf{B}^{\circ}$.
Shrinking $Z$ around $z$ if necessary, we can assume that the whole image of $Z$
in $\mf{B}$ is contained in the open subset $\mf{B}^{\circ}$.
Then by the construction of $\mf{B}^{\circ}$, the couple $(Y, D)$ is log smooth.  
In particular, the closed fibre $Y_z$ of $Y\to Z$ over $z\in Z$ is a smooth projective variety,
and $(Y, D + Y_z)$ is also a log smooth couple.
Write
\[ K_Y + B_Y = (\phi^{-1})^*(K_X + B^+), \]
where the coefficients of $B_Y$ are $\le 1$ as $(X, B^+)$ is lc; see Lemma~\ref{log-discrep-unchanged}.
Denote by $B_Y^h$ the horizontal$/Z$ part of $B_Y$.
By construction of the couple $(Y, D)$ in Lemma~\ref{bir-to-rlbnd-generic-sm-couples}, $B_Y^h\le D$,
hence up to shrinking $Z$ around $z$, we can assume that $B_Y \le D + Y_z$.
Thus, by Lemma~\ref{log-discrep-unchanged},
\[ 0\le a(F, Y, D + Y_z)\le a(F, Y, B_Y) = a(F, X, B^+) < 1, \]
which implies that $a(F, Y, D + Y_z) = 0$ by Proposition~\ref{rational-fibres};
in particular, we can conclude that $\centre_Y F$ is an lc centre of $(Y, D + Y_z)$ contained in the closed fibre $Y_z$,
so $\centre_Y F$ is a stratum of the log smooth projective couple $(Y_z, D_z)$.
Then it is evident that $\centre_Y F$ is bounded.
Moreover, as $(Y, D + Y_z)$ is a strict toroidal couple, 
we can conclude by applying Lemma~\ref{irr-bnd-above} and Proposition~\ref{rational-fibres},
and by taking $C$ to be a bounded smooth resolution of $\centre_{Y} F$. 

Denote by $\mf{Z}_{\mf{B}}\subset \mf{B}$ the proper closed subset $\mf{B}\setminus \mf{B}^{\circ}$.
From now on, we assume that the image of the closed point $z\in Z$
in $\mf{B}$ is contained in the closed subset $\mf{Z}_{\mf{B}}$.

\medskip

\emph{Step~2}.
By Theorem~\ref{functorial-toroidalization}, there is a commutative diagram
\[\xymatrix{
	  (U_{\mf{X}'}\subset \mf{X}')\ar[r]^-{m_{\mf{X}}}\ar[d]^{\Phi'} & \mf{X}\ar[d]^{\Phi} \\
	  (U_{\mf{B}'}\subset \mf{B}')\ar[r]^-{m_{\mf{B}}} & \mf{B}
}\]
such that $m_{\mf{B}}$ and $m_{\mf{X}}$ are projective birational morphisms, 
the inclusions on the left are strict toroidal embeddings, 
$m_{\mf{X}}^{-1}(\Supp \mf{D})$ is contained in $\mf{X}'\setminus U_{\mf{X}'}$, 
$m_{\mf{B}}^{-1}(\mf{Z}_{\mf{B}})$ is contained in $\mf{B}'\setminus U_{\mf{B}'}$,
and $\Phi'$ is a surjective projective toroidal morphism of strict toroidal embeddings.
Moreover, without loss of generality, we can assume that the open subset $U_{\mf{B}}$
in Theorem~\ref{functorial-toroidalization}~(ii) is contained in $\mf{B}^{\circ}$ of Step~1.

Let $Z\to \mf{B}$ be a morphism for a couple $(Y/Z, D)$ in 
$\mc{F}_{d, \epsilon, t}^{\text{sm}}$ which corresponds to some
$(X/Z, tF) \in \mc{F}_{d, \epsilon, t}$.
We can assume that $m_{\mf{B}}$ (respectively, $m_{\mf{X}}$)
is an isomorphism (respectively, is birational) over the generic point of 
the image of $Z$ in $\mf{B}$ and assume that the image of $Z$ intersects the open subset $U_{\mf{B}}\subset \mf{B}$;
otherwise, the image of $Z$ is entirely contained in a proper closed subset of $\mf{B}$,
then we can do induction on $\dim \mf{B}$ 
(which corresponds to taking stratification of $\mc{F}_{d, \epsilon, t}^{\text{sm}}$).

Denote by $g$ the morphism $Y\to Z$.
Let $Z'$ be the normalisation of the main component of $\mf{B}'\times_{\mf{B}} Z$.
Then $Z'\to Z$ is an isomorphism via $m_{\mf{B}}$ as $Z$ is a smooth curve.
Let $Y'$ be the normalisation of the main component of $\mf{X}'\times_{\mf{B}'} Z'$,
and let $U_{Y'}\subset Y'$ be the inverse image of $U_{\mf{X}'}$ in $Y'$.
Similarly, let $U_{Z'}\subset Z'$ be the inverse image of $U_{\mf{B}'}$ in $Z'$.
Denote by $h_{Y'}$ the induced morphism $Y'\to \mf{X}'$ and by 
$g'$ the morphism $Y'\to Z'$.
Then by Lemma~\ref{add-ver-to-toroidal}, it is evident that $U_{Y'} = h_{Y'}^{-1}(U_{\mf{X}'})$
is equal to $h_{Y'}^{-1}(U_{\mf{X}'}) \cap (g')^{-1}(U_{Z'})$.
By Theorem~\ref{functorial-toroidalization}~(ii), 
the induced morphism $g'\colon (U_{Y'}\subset Y') \to (U_{Z'}\subset Z')$
is also a toroidal morphism of strict toroidal embeddings.  
We include the commutative diagram here for convenience.
\[\xymatrix{
(U_{Y'}\subset Y')\ar[rrr]^-{m_{Y}}\ar[ddd]_{g'}\ar[rd]^-{h_{Y'}} & & & Y\ar[ddd]^{g}\ar[ld]  \\
 & (U_{\mf{X}'}\subset \mf{X}')\ar[d]_{\Phi'}\ar[r]^-{m_{\mf{X}}} & \mf{X}\ar[d]^{\Phi} &  \\
 & (U_{\mf{B}'}\subset \mf{B}')\ar[r]^-{m_{\mf{B}}} & \mf{B} &  \\
(U_{Z'}\subset Z')\ar[rrr]^-{m_Z}\ar[ru] & & & Z\ar[lu] 
}\]
Recall that there is a birational map $\phi\colon X\dashrightarrow Y$ over $Z$ and that $F$ is the 
reduction of the closed fibre of $X\to Z$ over the closed point $z\in Z$.
Moreover, by the construction in Lemma~\ref{bir-to-rlbnd-generic-sm-couples},
$\phi$ does not contract any horizontal$/Z$ divisors,
and all the horizontal$/Z$ exceptional divisors of $\phi^{-1}$ are contained in $\Supp D$.
Furthermore, by Theorem~\ref{functorial-toroidalization}~(ii),
$m_Y^{-1}(\Supp D)$ is contained in $Y'\setminus U_{Y'}$.

\medskip

\emph{Step~3}.
By Lemma~\ref{add-ver-to-toroidal} and the choice of the closed point $z\in Z$,  
the reduced divisor $\{z\}$ and
the support of the fibre of $g'$ over $z$
are contained in the toroidal boundaries of 
$(U_{Z'}\subset Z')$ and $(U_{Y'}\subset Y')$ respectively.  
Denote by $D'$ the reduced divisor supported on the complement $Y'\setminus U_{Y'}$.
Define the $\Q$-divisor $D_{Y'}$ by $K_{Y'} + D_{Y'} = (\phi^{-1}\circ m_Y)^* (K_X + B^+)$.
Similarly, denote by $D_Y$ the $\Q$-divisor defined by $K_Y + D_Y = (\phi^{-1})^*(K_X + B^+)$.
Note that $D_{Y'}$ and $D_Y$ may have irreducible components with negative coefficients.
Since $Y$ is generically smooth over $Z$, over the generic point of $Z$, we can write
\[ K_{Y'} + R = m_Y^* K_Y, \]
where $R\le 0$ is supported in the exceptional locus of $m_Y$, hence we have
\[ D_{Y'} = m_Y^* D_Y + R\le m_Y^* D_Y \]
over the generic point of $Z$.  
By construction of the reduced divisor $D$ on $Y$ in Lemma~\ref{bir-to-rlbnd-generic-sm-couples}, 
$\Supp D_Y$ is contained in $\Supp D$ over the generic point of $Z$.
Moreover, since $m_{\mf{X}}^{-1}(\Supp \mf{D})$ is contained in $\mf{X}'\setminus U_{\mf{X}'}$,
and since $U_{Y'} = h_{Y'}^{-1}(U_{\mf{X}'})$ and $D = \mf{D}\times_{\mf{B}} Z$,
by the commutative diagram in Step~2, 
$m_Y^{-1}(\Supp D)$ is contained in $\Supp D'$.
Hence we can conclude that 
$\Supp m_Y^* D_Y$ is contained in $\Supp D'$ over the generic point of $Z$.
Then as the coefficients of $D_{Y'}$ are $\le 1$, 
we have $D_{Y'}\le D'$ over the generic point of $Z$.
On the other hand, as the support of the fibre of $g'$ over $z\in Z$ is contained in the toroidal boundary $D'$,
shrinking $Z$ around $z$ if necessary, we can assume that
$D_{Y'}\le D'$ over the whole $Z$.

Then by Lemma~\ref{log-discrep-unchanged}, we have the relation of log discrepancies
\[ 0\le a(F, Y', D')\le a(F, Y', D_{Y'}) = a(F, X, B^+) < 1. \]
Since $(Y', D')$ is a strict toroidal couple, $K_{Y'} + D'$ is Cartier by Lemma~\ref{toroidal-normal-lc}, 
so $a(F, Y', D') = 0$.  That is, $\centre_{Y'} (F)$ is an lc centre of 
the pair $(Y', D')$.

\medskip

\emph{Step~4}.
Denote by $\mf{D}'$ the toroidal boundary $\mf{X}'\setminus U_{\mf{X}'}$.
Take a very ample$/\mf{B}'$ divisor $\mc{A}$ on $\mf{X}'$ so that
$K_{\mf{X}'} + \mf{D}' + \mc{A}$ is ample$/\mf{B}'$.
Denote by $A$ the pullback of $\mc{A}$ to $Y'$ which is an ample$/Z'$ Cartier divisor on $Y'$.
Denote by $S$ the reduction of the closed fibre of $Y'$ over $z\in Z$.
Motivated by the notation from \cite[Proposition~(5.10.17)]{EGA-IV-II}, we denote the $S_2$-isation
of $S$ by $S^{(1)}$.  As $S$ is contained in the divisor $D'$,
the adjunction in \cite[Proposition~16.6]{Corti-adjunction} shows that
the codimension one points of $S$ are either regular or
double normal crossings; see also \cite[Corollary~2.32]{Kol_singularities_of_MMP}.  Then $S^{(1)} \to S$
is a finite birational morphism which is an isomorphism over all codimension one points of $S$;
see \cite[Corollaire~(5.11.2)]{EGA-IV-II}.  Again by adjunction, we can write
$K_{S^{(1)}} + D^{(1)} = (K_{Y'} + D')|_{S^{(1)}}$
for some boundary divisor $D^{(1)}$, 
and the pair $(S^{(1)}, D^{(1)})$ is slc; cf. \cite[\S 4.1]{Kol_singularities_of_MMP}.
Denote by $A^{(1)}$ the pullback of $A$ to $S^{(1)}$.  
As $A^{(1)}$ is Cartier and ample, $K_{S^{(1)}} + D^{(1)} + (3\dim S)A^{(1)}$ is also ample
by the cone theorem of slc pairs; see \cite[Theorem~1.19]{Fujino-slc}.
Thus, replacing $\mc{A}$ by $3(d-1)\mc{A}$ if necessary, we can assume that 
the Cartier divisor $K_{Y'} + D' + A$ is also relatively ample over $Z'$.

\medskip

\emph{Step~5}.  
Now let $V$ be an lc centre of $(Y', D')$ 
that is a proper closed subset of an irreducible component of $S$.
We show that $V$ is birationally bounded.
Notice that a general closed fibre of $\Phi'$ is normal, hence
up to shrinking $U_{\mf{B}}\subset \mf{B}$ in Theorem~\ref{functorial-toroidalization}~(ii),
we can assume that the normal scheme $Y'$
is isomorphic to the fibre product $\mf{X}'\times_{\mf{B}'}Z'$ over the generic point of $Z'$.
Thus, by the construction in Step~4, there is an integer $r\in \N$ depending only on
$(\mf{X}', \mf{D}') \to \mf{B}'$ and $\mc{A}$ (hence depending only on $d, \epsilon, t$)
such that the volume
of $K_{Y'} + D' + A$ on a general closed fibre of $Y'\to Z'$ is less than $r$.
Let $W$ be the normalisation of $V$. By \cite[Theorem~1.1, Definition~1.3]{FH-adjunction}, 
we can write $(K_{Y'}+ D')|_W = K_W+D_W+M_W$, where $(W,D_W+M_W)$ is a generalised lc generalised pair. 
Now let $L=K_{Y'}+D' +A$, and 
let $L_W \coloneqq L|_W$ which is an ample Cartier divisor on $W$ as $L$ is Cartier and ample$/Z'$. 
Denote by $A_W$ the pullback of $A$ to $W$, which is also an ample Cartier divisor on $W$.
Then $(W, D_W + M_W + A_W)$ is also a generalised lc generalised pair
whose nef part is $M_W + A_W$.
By applying Lemma~\ref{b-bnd} to $L_W = K_W + D_W + M_W + A_W$, 
we see that $\abs{mL_W}$ defines a birational map
for some $m\in \N$ depending only on $\dim W$. 
It is then enough to show that the volume of $L_W$ is bounded 
by \cite[Lemma~2.4.2~(2)]{HMX13-auto-groups}. 

Pick an irreducible component of $S$ containing $V$ and let $T$ be its normalisation. 
Applying \cite[Theorem~1.1]{FH-adjunction}, we can write 
$(K_{Y'}+D')|_T = K_T+D_T$, where $(T,D_T)$ is lc.
(Since $\dim T = \dim Y' - 1$, there is no nef part of the generalised pair in this adjunction;
see \cite[Corollary~1.4]{FH-adjunction}.)
Take a general Cartier divisor $P\ge 0$ on $Y'$
that contains $V$ and does not contain other lc centres of $(Y', D')$ lying outside $V$.
For any small $\alpha>0$, $(Y', D'+ \alpha P)$ is not lc near $V$, hence
applying \cite[Theorem~1.1]{FH-adjunction} again, we deduce that 
$(T,D_T)$ has an lc centre $V_T$ mapping onto $V$. 
Indeed, by \cite[Theorem~1.1]{FH-adjunction}, we can write
\[ (K_{Y'} + D' + \alpha P)|_T = K_T + D_T + \alpha P_T, \]
where $P_T$ is the pullback of $P$ to $T$,
and $(T, D_T + \alpha P_T)$ is not lc for any $\alpha > 0$.  
Then as $(T, D_T)$ is lc, there exists an lc centre $Q_T$ of $(T, D_T)$ contained in $\Supp P_T$.
Taking $P$ generally, we see that the image of $Q_T$ in $Y'$ is contained in $V$.
Assume that no such lc centre $Q_T$ of $(T, D_T)$ is mapped onto $V$.
Then there is a proper closed subset $\Gamma\subset V$ such that the image
of each such $Q_T$ in $Y'$ is contained in $\Gamma$.
Applying \cite[Theorem~1.1]{FH-adjunction} and the argument as above to
the pair $(Y'\setminus \Gamma, D'|_{Y'\setminus \Gamma})$, we see that
there exists an lc centre of $(T, D_T)$ whose generic point has image in $Y'$ contained in $V\setminus \Gamma$.
This is a contradiction by construction of $\Gamma\subset V$, hence 
there is an lc centre $V_T$ of $(T, D_T)$ mapping onto $V$.
    
Denote by $A_T$ the pullback of the ample$/Z'$ Cartier divisor $A$ to $T$.
By assumption on volumes, we have that
the volume of the ample Cartier divisor $K_T + D_T + A_T = (K_{Y'} + D' + A)|_T$ is $\le r$,
so the volume of $K_T + D_T + A_T$ takes only finitely many values.
Changing the ample Cartier divisor $A_T$ linearly, we can assume that $(T, D_T + A_T)$ is lc.
Then \cite[Theorem~1.1]{HMX14} shows that $(T, D_T + A_T)$
belongs to a bounded set of pairs, so 
the lc centre $V_T$ of $(T, D_T)$ is also bounded by Lemma~\ref{lc-centre-bnd}.
This implies that the volume of $(K_{Y'}+D'+A)|_{V_T} = (K_T + D_T + A_T)|_{V_T}$ is bounded, hence the volume of 
$L_W = (K_{Y'}+D'+A)|_W$ is bounded as desired.

\medskip
    
\emph{Step~6}.  Now we can conclude the proof as follows.  
Recall that $\centre_{Y'}(F)$ is an lc centre of the pair $(Y', D')$ that is contained in $S$.
By Step~5, $\centre_{Y'}(F)$ is birationally bounded.
Denote by $C$ a bounded nonsingular resolution of $\centre_{Y'}(F)$.
Hence the irrationality $\irr (C)$ 
is bounded from above by Lemma~\ref{irr-bnd-above}.
Moreover, by Proposition~\ref{rational-fibres}, 
$F\bir \centre_{Y'}(F)$ has irreducible and rational general closed fibres,
so does the rational map $F\dashrightarrow C$.
\end{proof}


\begin{cor}[=Theorem~\ref{bnd-irra-cor}]
    Fix a positive real number $\epsilon>0$ and a natural number $d$.
    Let $f\colon X\to Z$ be a Fano fibration with $\dim X = d$ such that
    \begin{itemize}
        \item $Z$ is a smooth curve, and
        \item $X$ is $\epsilon$-lc.
    \end{itemize}
    Let $F$ be an irreducible component of a closed fibre of $f$.
    Then there is a dominant rational map $F \dashrightarrow C$ whose general fibres are irreducible and rational,
    and $C$ is a bounded smooth projective variety depending only on $d,\epsilon$,
    hence with bounded degree of irrationality.
\end{cor}

\begin{proof}
    Up to decreasing $\epsilon$ slightly, 
    we can take an effective $\Q$-divisor $B$ so that $(X, B)$ is $\epsilon$-lc,
    and $K_X + B \sim_{\Q}0/Z$.  
    By taking a $\Q$-factorial dlt model of $(X, B)$, we can assume that 
    $X$ is $\Q$-factorial, and that the contraction $f\colon X\to Z$ is of Fano type;
    see \cite[\S 2.13~(7)]{B-Fano}.
    Then by \cite{BCHM}, we can run MMP over $Z$ on any $\R$-divisor on $X$.

    Let $z\in Z$ be the closed point such that $F$ is contained in the fibre $f^{-1}(z)$.
    By \cite{BCHM}, we can run a $(-F)$-MMP $\phi\colon X\dashrightarrow X'$ over $Z$
    that ends with a good minimal model of $-F$.
    By negativity lemma, $F$ is not contracted by $\phi$.
    Let $F' \coloneqq \phi_* F$.  Then $F'$ is supported on the whole fibre of $X'\to Z$ over $z$.
    As $F$ is vertical$/Z$, $\phi$ only modifies the fibre $f^{-1}(z)$,
    and leaves everything outside $f^{-1}(z)$ unchanged, so $-K_{X'}$ is big$/Z$.
    Again by \cite{BCHM}, we can run a $(-K_{X'})$-MMP over $Z$,
    which induces a Fano fibration $g\colon X''\to Z$.
    Denote by $F''$ the pushdown of $F'$ to $X''$.
    As $F'$ is supported on the whole fibre of $X'\to Z$ over $z$, we see that $F''$
    is an irreducible divisor supported on the whole fibre of $g$ over $z$.
    
    Let $B''$ be the pushdown of $B$ to $X''$.
    Then $(X'', B'')$ is also $\epsilon$-lc, and $K_{X''} + B''\sim_{\Q}0/Z$. 
    By the canonical bundle formula \cite[\S 2.11]{birkar2025toroidaltoricmodelsfibrations},
    we can write $K_{X''} + B'' \sim_{\Q} g^*(K_Z + B_Z + M_Z)$,
    where $(Z, B_Z + M_Z)$ is a generalised pair.
    By \cite[Theorem~1.1]{birkar2025singularitiesfanofibrations}, 
    there is a $\delta>0$ depending only on $d, \epsilon$ such that
    $(Z, B_Z + M_Z)$ is generalised $\delta$-lc.
    Then the coefficient of $z$ in $B_Z$ is $\le 1-\delta$.
    Thus, the lc threshold $\alpha$ of $g^*z$ with respect to $(X'', B'')$ 
    is bounded from below: indeed, the coefficient of $z$ in $B_Z$ is $1-\alpha$,
    which implies that $\alpha\ge \delta$.
    Then there exists a $t>0$ depending only on $d, \epsilon$ such that $(X'', B'' + t F'')$ is lc,
    hence $(X'', t F'')$ is lc.  Thus, we can conclude by Theorem~\ref{bnd-irra}.
\end{proof}


\medskip

\noindent\small{Caucher Birkar} 

\noindent\small{\textsc{Yau Mathematical Sciences Center, Tsinghua University, Beijing, China} }

\noindent\small{Email: \texttt{birkar@mail.tsinghua.edu.cn}}

\vspace{1em}
 
\noindent\small{Santai Qu} 

\noindent\small{\textsc{Institute of Geometry and Physics, University of Science and Technology of China, Hefei, Anhui Province, China} }

\noindent\small{Email: \texttt{santaiqu@ustc.edu.cn}}



\end{document}